\documentclass[12pt, longbibliography,leqno]{amsart}

\usepackage[utf8]{inputenc}
\usepackage[right=3cm,left=3cm,textheight=22cm]{geometry}
\usepackage{amsmath}
\usepackage{amsthm}
\usepackage{amssymb}
\usepackage{amsfonts}
\usepackage{mathtools}
\usepackage{paralist}
\usepackage[colorlinks=true,linkcolor=blue,citecolor=blue,urlcolor=blue,breaklinks]{hyperref}
\usepackage{url}
\usepackage{cases}
\usepackage{enumitem}
\usepackage{empheq}
\usepackage{multimedia}
\usepackage{cite}

\allowdisplaybreaks
\mathtoolsset{showonlyrefs}     % Number only eqs which are referred to

\numberwithin{equation}{section}

\theoremstyle{plain}
\newtheorem{thm}{Theorem}[section]

\newtheorem{lem}[thm]{Lemma}
\newtheorem{prop}[thm]{Proposition}
\newtheorem{hyp}{Hypothesis}

\newtheorem{subhyp}{Hypothesis}

\theoremstyle{definition}
\newtheorem{defn}[thm]{Definition}

\theoremstyle{remark}
\newtheorem{rem}[thm]{Remark}

\newcommand{\be}{\begin{equation}}
\newcommand{\ee}{\end{equation}}
\newcommand{\bes}{\begin{equation*}}
\newcommand{\ees}{\end{equation*}}
\newcommand{\bfig}{\begin{figure}}
\newcommand{\efig}{\end{figure}}
\newcommand{\bt}{\begin{table}}
\newcommand{\et}{\end{table}}
\newcommand{\bc}{\begin{center}}
\newcommand{\ec}{\end{center}}
\newcommand{\ba}{\begin{array}}
\newcommand{\ea}{\end{array}}

\newcommand{\abs}[1]{\left\vert#1\right\vert}

\newcommand{\norm}[1]{\left\Vert#1\right\Vert}
\newcommand{\mt}[1]{\mathrm{#1}}

\newcommand\numberthis{\addtocounter{equation}{1}\tag{\theequation}}

\newcommand{\floor}[1]{\left\lfloor#1\right\rfloor}
\newcommand{\ceil}[1]{\left\lceil#1\right\rceil}

\def\R{\mathbb{R}}
\def\N{\mathbb{N}}
\def\P{\mathcal{P}}

\def\M{\mathcal{M}}
\def\C{\mathcal{C}}
\def\e{\varepsilon}

\def\d{\,\mathrm{d}}

\def\ird{\int_{\R^d}}

\def\st{\, \left|\right. \,}

\def\oneN{\{1,\dots,N\}}
\def\:{\colon}

\def\E{E_N}

\def\bX{\boldsymbol{X}}
\def\bXN{\boldsymbol{X_N}}
\def\bY{\boldsymbol{Y}}
\def\bYN{\boldsymbol{Y_N}}
\def\bXNs{\boldsymbol{X_N^*}}

\def\Np{N_{\mathrm{p}}}
\def\Ne{N_{\mathrm{e}}}
\def\mup{\mu_N^{\mathrm{p}}}
\def\mue{\mu_N^{\mathrm{e}}}

\def\Xint#1{\mathchoice
{\XXint\displaystyle\textstyle{#1}}%
{\XXint\textstyle\scriptstyle{#1}}%
{\XXint\scriptstyle\scriptscriptstyle{#1}}%
{\XXint\scriptscriptstyle\scriptscriptstyle{#1}}%
\!\int}
\def\XXint#1#2#3{{\setbox0=\hbox{$#1{#2#3}{\int}$ }
\vcenter{\hbox{$#2#3$ }}\kern-.6\wd0}}

\def\dashint{\Xint-}

\DeclareMathOperator{\dist}{dist}
\DeclareMathOperator{\supp}{supp}

\DeclareMathOperator{\Lip}{Lip}

\DeclareMathOperator{\diam}{diam}

\begin{document}

%================== TITLE =========================================================
\title[Discrete minimisers of the interaction energy]{Discrete
  minimisers are close to continuum minimisers for the interaction
  energy}

\author{J. A. Ca\~nizo}
\address{Departamento de Matemática Aplicada, Facultad de Ciencias. Avenida de Fuente-nueva S/N, 18071 Granada, Spain}
\email{canizo@ugr.es}

\author{F. S. Patacchini}
\address{Department of Mathematical Sciences, Carnegie Mellon University, Pittsburgh, PA 15203, USA}
\email{f.patacch@math.cmu.edu}

\date{December 2017}
\subjclass[2010]{35A15, 35Q70, 49M25, 82B21.}

\maketitle

\begin{abstract}
  Under suitable technical conditions we show that minimisers of the
  discrete interaction energy for attractive-repulsive potentials
  converge to minimisers of the corresponding continuum energy as the
  number of particles goes to infinity. We prove that the discrete
  interaction energy $\Gamma$-converges in the narrow topology to the
  continuum interaction energy. As an important part of the proof we
  study support and regularity properties of discrete minimisers: we
  show that continuum minimisers belong to suitable Morrey spaces and
  we introduce the set of empirical Morrey measures as a natural
  discrete analogue containing all the discrete minimisers.
\end{abstract}

%================== TABLE OF CONTENTS ============================================
\tableofcontents
%\vspace{0.4cm}

\section{Introduction}
\label{sec:introduction}

Consider a finite set of $N \geq 2$ classical particles in Euclidean
space $\R^d$ interacting through a pair potential
$W \: \R^d \to \R \cup \{+\infty\}$. If the particles are placed at
points $x_1, \dots, x_N \in \R^d$ and have equal masses $1/N$, then
their total interaction (potential) energy is given by
\begin{equation}
  \label{eq:discrete-energy}
  E_N(\bX) :=
  \frac{1}{2 N^2} \sum_{i=1}^N \sum_{\substack{j=1\\j\neq i}}^N W(x_i-x_j),
\end{equation}
where we denote $\bX \equiv (x_1, \dots, x_N)$. We often refer to a
set of positions $(x_1, \dots, x_N) \in \R^{Nd}$ as a
\emph{configuration}, and call $E_N$ the \emph{discrete interaction
  energy}. The gradient of $W$ models the interaction force between
two particles: $-m_x m_y \nabla W(x-y)$ is the force that a particle
at $x$ with mass $m_x$ exerts on a particle at $y$ with mass $m_y$,
and accordingly we say that $W$ is \emph{attractive} at $x \in \R^d$
when $-\nabla W(x) \cdot x \leq 0$ and \emph{repulsive} when
$-\nabla W(x) \cdot x \geq 0$. The choice of masses equal to $1/N$ is
of course a convenient normalisation so that the set of $N$ particles
has total mass 1. Notice that $W$ can be assumed to be symmetric,
i.e., $W(-x) = W(x)$ for all $x\in\R^d$, without loss of generality
since symmetrising the potential does not change the energy
\eqref{eq:discrete-energy}. The sum in \eqref{eq:discrete-energy} is
therefore halved since pairs of particles are counted twice; more
importantly, self-interactions are not present in the sum, in
agreement with classical physics.

A natural question regards the existence and shape of
\emph{minimisers} of this interaction energy among all possible
particle configurations; that is, particle configurations whose
interaction energy is the smallest possible. We also refer to these
configurations as \emph{ground states} or \emph{discrete minimisers},
and in this paper we are mainly concerned with their shape and size as
$N \to \infty$. As we rigorously show, for large $N$ these minimisers
are closely related to minimisers of the \emph{continuum interaction
  energy} $E$ defined by
\begin{equation}
  \label{eq:energy}
  E(\rho) = \dfrac{1}{2} \ird\ird W(x - y) \d\rho(x) \d\rho(y)
\end{equation}
for any $\rho\in\P(\R^d)$, where $\P(\R^d)$ is the set of Borel
probability measures on $\R^d$. This expression makes sense whenever
$W$ is bounded from below and measurable, in which case the value of
$E(\rho)$ is a number in $\R \cup \{+\infty\}$. A probability measure
$\rho$ minimising \eqref{eq:energy} on $\P(\R^d)$ is called a
\emph{continuum minimiser}. These continuum minimisers have been the
subject of several works \cite{SST,CCP, CFT15, BCLR13,BCY14} and in
particular their existence, under some technical assumptions, is
almost equivalent to the \emph{instability} of the potential $W$
\cite{SST,CCP}:

\begin{defn}[Instability]
  \label{defn:instability}
  Let $W \: \R^d \to \R \cup \{+\infty\}$ be a measurable function
  which is bounded from below and such that there exists
  \begin{equation*}
    W_\infty := \lim_{|x| \to \infty} W(x) \in \R\cup\{+\infty\}.
  \end{equation*}
  We say that $W$ is \emph{unstable} if there exists $\rho \in \P(\R^d)$ such that 
  \begin{equation*}
    E(\rho) < \textstyle{\frac12} W_\infty.  
  \end{equation*}
  We say that $W$ is \emph{stable} if it is \emph{not} unstable, i.e.,
  if for all $\rho\in\P(\R^d)$ we have $E(\rho) \geq W_\infty/2$; we
  say that $W$ is \emph{strictly stable} if the inequality is strict.
\end{defn}

This concept of stability is very close to the classical concept of
$H$-stability as given for example in
\cite{Ruelle}; see Definition \ref{defn:H-instability}. For
continuous potentials it was proved in \cite{SST} that they are
indeed equivalent, and for potentials with a mild singularity at the origin
we show in Section \ref{subsec:convergence-discrete} that
$H$-instability is implied by instability as given in the above
definition. We refer the reader to Section \ref{subsec:convergence-discrete} for
further details and a brief background of these concepts.
In \cite{CCP,SST} it was proved that under some technical assumptions
(for example, under Hypotheses \ref{hyp:bfb} and \ref{hyp:existence}
in the next section) continuum minimisers exist if and only if there
exists a probability measure $\rho$ with $E(\rho) \leq W_\infty/2$;
that is, if and only if W is unstable or there is $\rho$ with
$E(\rho) = W_\infty/2$. It is then natural to find that the same
concept of instability plays a crucial role in the behaviour of
discrete minimisers for large $N$---this constitutes the main
result in our paper: \emph{for unstable potentials, discrete
  minimisers approach the set of continuum minimisers as
  $N \to \infty$; while for strictly stable potentials, discrete
  minimisers grow in size without bound as $N \to \infty$}.

In order to state the main result precisely we need a few
definitions. The \emph{diameter} of a particle configuration
$\bX = (x_1,\dots,x_N) \in \R^{Nd}$, denoted $\diam \bX$, is just the diameter of the set
$\{x_1, \dots, x_N \} \subseteq \R^d$ and the \emph{empirical
  measure} associated to $\bX$ is
\begin{equation*}
  \mu_{\bX} := \frac1N\sum_{i=1}^N \delta_{x_i},
\end{equation*}
with $\delta_{x}$ the Dirac delta measure at a point $x \in \R^d$. We endow the set $\P(\R^d)$ with
the \emph{narrow} topology, obtained by duality with the space of
bounded continuous functions on $\R^d$, in accordance with the terminology in \cite{AGS}; we discuss more the narrow topology in Section \ref{sec:gamma}. We often identify
$\bX \in \R^{Nd}$ with its empirical measure $\mu_{\bX}$, and
accordingly we use on $\bX$ concepts that really apply to
$\mu_{\bX}$. For example, we say that a sequence $(\bXN)_{N\geq2}$,
with $\bXN \in \R^{Nd}$ for all $N\geq2$, converges to $\rho$
(narrowly) if $\mu_{\bXN}$ converges to $\rho$ in the narrow
topology. Postponing the precise hypotheses to Section
\ref{sec:hypotheses}, the following is our main result:

\begin{thm}[Main result]
  \label{thm:main}
  Assume that $W$ satisfies Hypotheses \ref{hyp:bfb},
  \ref{hyp:existence} and \ref{hyp:W-close-to-0} in Section
  \ref{sec:hypotheses}. For any $N\geq2$ the discrete energy $\E$ has
  a minimiser on $\R^{Nd}$, and for any sequence $(\bXN)_{N\geq2}$ of
  such minimisers the following statements hold:
  \begin{enumerate}
  \item\label{it:main-bounded} If $W$ is unstable, then the diameter of $\bXN$ is
    uniformly bounded for all $N$ and $(\bXN)_{N\geq2}$ has a
    subsequence which converges in the narrow sense, up
    to translations, to a minimiser of the continuum energy $E$
    as $N\to\infty$.
    
  \item\label{it:main-unbounded} If $W$ is strictly stable, then the diameter of
    $\bXN$ tends to $\infty$ as $N \to \infty$.
  \end{enumerate}
\end{thm}

The case in which
$\inf_{\rho\in\P(\R^d)} E(\rho) = \min_{\rho\in\P(\R^d)} E(\rho) =
(1/2)\lim_{|x| \to \infty} W(x)$, that is, the case in which $W$ is stable
but not strictly stable, is not included in our main result. Indeed this is a critical case for which our
approach is not conclusive; discrete minimisers may still exist for
every $N$ but we cannot get a uniform bound on their diameter, which
prevents us from proving convergence to a continuum minimiser. The
precise hypotheses on $W$ are given in Section \ref{sec:hypotheses}
but we already point out that the \emph{power-law} potentials
\begin{equation}
  \label{eq:power-law}
  W(x) =
  \dfrac{|x|^a}{a} - \dfrac{|x|^b}{b}
  \quad \text{with}
  \quad
  \begin{cases}
    0 < b < a & \text{when $d\in\{1,2\}$},
    \\    
    2-d < b < a,\quad b \neq 0, & \text{when $d \geq 3$},
  \end{cases}
\end{equation}
satisfy all assumptions in the main theorem and are unstable (and thus their associated
discrete minimisers behave as Theorem
\ref{thm:main}\eqref{it:main-bounded}). The \emph{Morse potentials}
\begin{equation}
  \label{eq:Morse}
  W(x) = C_\mt{r} \exp(-|x|/\ell_\mt{r})
  - C_\mt{a} \exp(-|x|/\ell_\mt{a}),
\end{equation}
for some positive constants $C_\mt{r},\ell_\mt{r},C_\mt{a}$ and
$\ell_\mt{a}$, also satisfy all assumptions and are unstable if $\ell_\mt{r} < \ell_\mt{a}$ and $C_\mt{r}/C_\mt{a} < (\ell_\mt{a}/\ell_{r})^d$; see \cite[Proposition 3.2]{CCP}.

Theorem \ref{thm:main} gives a natural link between the discrete
energy \eqref{eq:discrete-energy} and the continuum one
\eqref{eq:energy}. In the same way as \cite{CCP,SST} showed that strict
stability is the property of the potential that determines existence
or not of continuum minimisers, Theorem \ref{thm:main} shows that it
also determines the existence or not of a limit of the family of
discrete minimisers as $N \to \infty$. In fact our proof follows
``discrete versions'' of arguments in \cite{CCP}. Regularity results
for continuum minimisers in \cite{BCLR13} are also crucial and our
proof contains discrete analogues of these. Let us now give
some background motivation for the problem we are considering and
then let us sketch the strategy of our proof.

\medskip
\noindent \textbf{Previous results and motivation.} Discrete
minimisers represent the natural minimal energy configurations of $N$
particles under the given potential in the absence of any external
forces and without thermal fluctuations (in other words, these are
classical ground states at temperature 0). Understanding the shape of
these ground states (and those of related energy functionals) when the
number $N$ of particles is very large is of obvious interest in
statistical mechanics \cite{Serfaty2016,PS2015,Theil}, with direct
implications in materials science
\cite{MPS2014,MS2014,MPS2015,GPPS2013}. For physically relevant
potentials such as the Lennard-Jones potential
$W(x) = |x|^{-12} - |x|^{-6}$ the conjectured behaviour is that
\emph{crystallisation} takes place as $N \to \infty$. That is:
minimisers have particles placed almost at the vertices of a regular
triangular lattice, approaching the lattice as $N$ increases. This has
been rigorously proved for certain potentials similar to Lennard-Jones
in dimensions 1 and 2 \cite{Theil,GR} and for some other very specific
potentials \cite{MPS2014,MS2014}, but is in general an unsolved
problem; for results in this direction, see also \cite{PS2015} for some interaction energies with an external confining potential
and \cite{Serfaty2016} for systems with Coulomb interactions and their
links to other mathematical problems. Even if showing a
crystallisation property is remarkably hard, one can make a weaker
observation: for certain potentials, including Lennard-Jones, the
\emph{diameter} of ground states seems to increase without bound as
$N \to \infty$, while for others the diameter seems to tend to a fixed
value. This is part of the content of Theorem \ref{thm:main}, whose
main restriction in this setting is that it essentially requires the
potential $W$ to be less singular than $|x|^{2-d}$ at $x=0$ (such as
for example \eqref{eq:power-law}). When the singularity is stronger,
between $|x|^{-d}$ and $|x|^{2-d}$, we expect our main result still to
be true, although we are unable to show it since potentials in this
case do not satisfy Hypothesis \ref{hyp:W-close-to-0}. Hence our
statement does not say anything about the Lennard-Jones case (although the concept of stability still makes sense, and in fact the Lennard-Jones potential is stable), but does
show that minimisers grow in diameter without bound for a range of
stable potentials with a possible singularity at $x=0$.

In addition to their relevance in statistical mechanics, an
important more recent motivation for our results comes from the field
of collective behaviour, where shapes of self-organised structures in
some individual-based models exhibit very interesting phenomena and
are closely related to those of discrete minimisers \cite{DCBC,
  KSUB11, ABCvB14, CHM14, BKSUB15, BUKB2012}; see the survey on emergent behaviour \cite{KCBFL} and the references therein. In this context, models
aim at capturing the behaviour of a large number of individuals, with
applications to fish, cattle, birds, ants, and crowds of people. In
very simplified models, interaction through a potential reflects a
tendency of individuals to avoid close contact while keeping a
tendency to stay close to the group. The study of these models has led
to different questions regarding the minimisers of
\eqref{eq:discrete-energy}, mainly since the potentials involved are
not determined by physics but by phenomenological considerations in
each particular model. This has sparked interest in the shape of
minimisers for potentials which are very different from those found in
physics, including potentials with a mild or no singularity at 0 or
which tend to infinity at large distances. The paper \cite{DCBC} is
the first example we know of where the link was made between the
stability properties of the potential and the size of stationary
states for a potential interaction. In it, a particular time-dependent
interaction model was considered with the Morse potential
\eqref{eq:Morse} and its asymptotic states were numerically
studied. It was observed that their size increases with $N$ for stable
potentials while it does not for unstable ones. This is precisely the behaviour which Theorem \ref{thm:main} aims at justifying
rigorously.

Minimisers of the continuum energy \eqref{eq:energy} are also of
interest in collective behaviour models
\cite{Fetecau2012Equilibria,Fetecau2011Swarm,CCH2}, in the theory of
nonlocal partial differential equations
\cite{CSV13,CV11,citeulike:12903864}, and again in connection to
statistical mechanics \cite{Bavaud}. They display interesting effects
such as a link between the repulsive singularity of the potential and
smoothness of minimisers
\cite{BCLR13,Fellner2010Stability,Fellner2010Stable,CSV13}; they are
connected to solutions to obstacle problems in certain cases
\cite{CDM2016,CV}; and for specific potentials $W$ they are also linked
to the theory of random matrices \cite{CGZ2014}. These continuum
minimisers are often studied by numerically solving an $N$-particle
approximation, with the assumption that stationary states for large
$N$ are good approximations to the continuum ones. As far as we know,
our present results are the first where a justification of this is
given. Of course, in order to make the results practical for numerical
simulation it would be very useful to estimate the rate of
convergence to continuum minimisers as $N \to \infty$ in Theorem
\ref{thm:main}; this seems a worthwhile but difficult question, since
even the uniqueness of minimisers is unclear (except for specific
potentials $W$ \cite{Fetecau2012Equilibria,CV11}).

Let us mention as well that the connection between the discrete and
continuum energies is a hard question in mean-field limit results for
dynamical problems \cite{HJ2007,CCHS2016,CCH,CDFLS,CCR11}, especially
for potentials which have a singularity at $x=0$. Roughly speaking,
the main difficulty is to show that the unbounded forces between
particles resulting from the singularity are in fact negligible for
large $N$ if one only cares about the overall particle
density. Unsurprisingly, our proof is much more delicate for singular
potentials and yields more interesting estimates at the discrete
level in that case.

\medskip
\noindent \textbf{Strategy of proof.} Our general strategy is based on
drawing a parallel discrete version of several results which have been
recently obtained for continuum minimisers. A first one is the
\emph{regularity} of continuum minimisers, studied in
\cite{BCLR13,CSV13,CDM2016}. We describe this informally now and we
refer the reader to Section \ref{sec:prop-continuum} for full details. If the
potential $W$ behaves like $-|x|^{b}/b$ close to $x = 0$ for some
$2-d < b < 2$, then it was proved in \cite{BCLR13} that the dimension of
the support of continuum minimisers is at least $2-b$. In fact, a
stronger regularity result is a direct consequence of the arguments in
\cite{BCLR13}, though it was not explicitly remarked there: it holds
that $|x|^{b-2} * \rho$ is a bounded function for any minimiser
$\rho$ and hence one obtains (see Section \ref{sec:prop-continuum})
that $\rho$ is in the Morrey space of measures which
satisfy $$\rho(B_r) \leq C r^{2-b}$$ for any ball $B_r$ of radius
$r > 0$, and some $C > 0$ independent of the ball $B_r$. Now, an
analogue of this regularity is needed for discrete minimisers
$\bX = (x_1,\dots,x_N) \in \R^{Nd}$, with the difficulty that the
empirical measure
$$\mu_{\bX} := \frac1N\sum_{i=1}^N \delta_{x_i}$$ cannot
satisfy the same bound, being a sum of Dirac delta functions. Instead
of this, we prove the following variation: the total mass of
particles inside a ball of radius $r$ \emph{centred at one of the
  particles} is less than a constant times $r^{2-b}$ if one does not
count the particle at the center. In other words, there exists a universal
constant $C > 0$ depending only on $W$ such that
\begin{equation}
  \label{eq:discrete-Morrey-intro}
  \mu_{\bX} \big( B_r(x_i) \big)
  \leq C r^{2-b} + \tfrac1N
  \quad \text{for $i\in\{1, \dots, N\}$}
\end{equation}
for any discrete
minimiser $\bX = (x_1,\dots,x_N) \in \R^{Nd}$ (always under the assumption that $W(x) \sim -|x|^b/b$ for $x \sim
0$).
This motivates an interesting definition of ``empirical Morrey
measures'' which serves as a discrete version of the Morrey spaces; see Sections \ref{sec:prop-continuum} and \ref{sec:prop-discrete}
for details.

Another important basic property of continuum minimisers is that they
satisfy the following conditions, as proved in \cite{BCLR13} (and
informally noticed in \cite{Bavaud} without a rigorous proof): if a
probability measure $\rho$ minimises \eqref{eq:energy} then
\be\label{eq:ELcont}
  \begin{cases}
    W*\rho(x) = 2E(\rho) & \mbox{for $\rho$-almost every $x\in\R^d$},\\
    W*\rho(x)\geq 2E(\rho) & \mbox{for almost every $x\in\R^d$}.
  \end{cases}
\ee
The quantity $W * \rho(x)$ represents the potential created by the mass
distribution $\rho$ at the point $x \in \R^d$; the above statement
says in particular that it is almost everywhere constant in the
support of $\rho$. We refer to the condition in the first line as the
\emph{Euler--Lagrange} equation. The quantity corresponding to
$W*\rho(x)$ in the discrete case, for a particle distribution
$\bX \in \R^{Nd}$, is
\begin{equation*}
  P_i(\bX) := \frac1N\sum_{\substack{j=1\\j\neq i}}^N W(x_i-x_j)
  \quad \mbox{for all $i\in\oneN$},
\end{equation*}
which is the potential at position $x_i$ created by all particles but that at $x_i$. Interestingly, for a discrete minimiser
this does \emph{not} seem to be constant at all sites $i$, but we show
a bound on its variation across sites which decays asymptotically as
$N \to \infty$: there exist $A > 0$ and $0 < k \leq 1$ such that
\begin{equation}
  \label{eq:EL}
  |P_i(\bX) - P_j(\bX)| \leq AN^{-k}
  \quad \mbox{for all $i,j\in\oneN$}
\end{equation}
for any discrete minimiser $\bX \in \R^{Nd}$. The constants $A$ and
$k$ are independent of $N$ (and are constructive) and thus this shows
that for large $N$ the potential at two different particles cannot
differ by a large amount.

Finally, continuum minimisers are known to be compactly supported if $W$ is increasing at long range,
with a constructive bound as proved in \cite{CCP}. Analogously, in
Section \ref{sec:diameter-discrete} we give a uniform bound on the
diameter of discrete minimisers, which can be understood as a discrete
version of the argument in \cite{CCP}, using the approximate
Euler--Lagrange property \eqref{eq:EL} and the ``discrete Morrey
regularity'' \eqref{eq:discrete-Morrey-intro}. We point out that the
latter is needed only for potentials which are unbounded at $x=0$,
which are the main difficulty in our result.

We also phrase some of our results using the terminology of
$\Gamma$-convergence in Section \ref{sec:gamma}. Our proof of
convergence of minimisers contains the fact that the discrete energy
\eqref{eq:discrete-energy} $\Gamma$-converges to the continuum energy
\eqref{eq:energy} in the narrow topology, which depends on
the singularity of the potential $W$. We remark that there is a
previous related result in \cite{CT2016}, where the
$\Gamma$-convergence of the \emph{regularised} continuum energy
(associated to a mollified potential $W_\epsilon$) to the energy
\eqref{eq:energy} (associated to $W$) was studied as the
regularisation parameter $\epsilon$ tends to $0$. Hence this latter
result is concerned with convergence of the continuum energy
\eqref{eq:energy} for different potentials, while in the present
paper we study the convergence of the discrete energy
\eqref{eq:discrete-energy} as $N \to \infty$ for a fixed potential
$W$.

\medskip The structure of the paper is as follows. In Section
\ref{sec:hypotheses} we gather some necessary definitions and state
precisely our hypotheses. Sections \ref{sec:prop-continuum} contains
some simple observations on the regularity of continuum minimisers,
directly deduced from \cite{BCLR13}. Section \ref{sec:prop-discrete}
gathers several properties of discrete minimisers, including
existence, ``discrete regularity'' (a discrete version of the
continuum one) and an approximate Euler--Lagrange property. Finally,
in Section \ref{sec:convergence-minimisers} the proof of our main
result is completed, showing that discrete minimisers approach the set
of continuum ones as the number of particles goes to infinity. In
Section \ref{sec:gamma} we prove a technical result that is needed in
earlier proofs: the discrete energy $\Gamma$-converges to the
continuum one; essentially, we show that one may approximate a
probability measure $\rho$ by a discrete distribution in such a way
that the interaction energy is also approximated if $\rho$ has a
suitable Morrey regularity.

\section{Preliminaries and hypotheses}
\label{sec:hypotheses}

In order to state the full assumptions in our results we need to
introduce a couple of definitions. The first is the concept of
$\beta$-repulsivity, taken from \cite{BCLR13}. For any $R\geq0$ and
$z\in\R^d$ we denote by $B_R(z)$ the open ball of radius $R$ and
centre $z$; in the case $z=0$ we simply write $B_R$. Analogously, we
write $\overline{B}_R(z)$, or $\overline{B}_R$, for the closed
ball. The integral $\dashint_A$ denotes the averaged integral over a
region $A$, that is, $\int_A$ divided by the Lebesgue measure of $A$.

\begin{defn}[Approximate and generalised Laplacians]
  \label{defn:laplacians} 
  Let $W \: \R^d \to \R \cup \{+\infty\}$ be a locally integrable
  function. The \emph{approximate Laplacian} of $W$ is defined, for
  all $\e>0$, by
  \begin{equation*}
    \Delta^\e W(x) = \frac{2(d+2)}{\e^2}\left(\dashint_{B_\e} W(x+y) \d y - W(x)\right)
    \quad \mbox{for all $x\in\R^d$},
  \end{equation*}
  and the \emph{generalised Laplacian} of $W$ is
  defined by
  \begin{equation*}
    \Delta^0 W(x) = \liminf_{\e \to 0} \Delta^\e W(x)
    \quad \mbox{for all $x\in\R^d$}.
  \end{equation*}
\end{defn}

Note that $\Delta^\varepsilon W$ makes sense as a number in
$\R \cup \{+\infty\}$, and $\Delta^0 W$ may be a number in
$\R \cup \{-\infty, +\infty\}$. Also, if the classical Laplacian of
$W$ exists at some $x\in\R^d$, then $\Delta W(x) = \Delta^0 W(x)$.

\begin{defn}[$\beta$-repulsivity]
  \label{defn:beta-repulsive}
  Let $\beta > 0$ and $W \: \R^d \to \R \cup \{+\infty\}$. We say that
  $W$ is \emph{$\beta$-repulsive} at the origin if it is locally
  integrable and there exist $\delta > 0$ and $C > 0$ such that
  \begin{equation}
    \label{eq:beta-rep} -
    \Delta^0W(x) \begin{cases} \geq C|x|^{-\beta} & \mbox{for all
        $x \in \R^d$ with $0 < |x| < \delta$},\\ = +\infty &
      \mbox{for $x = 0$}. \end{cases}
  \end{equation}
\end{defn}
Notice that the notion of $\beta$-repulsivity is sensitive to the
value of $W$ at $x=0$, so it does not hold if we arbitrarily set
$W(0) := W_0 \in \R$ when $W$ is lower semicontinuous (the second line of \eqref{eq:beta-rep} would not
be satisfied). Typically, potentials with a singularity equal to or
stronger than the Newtonian are generally \emph{not} $\beta$-repulsive
for any $\beta > 0$.  Indeed, if $W(x) := -|x|^b/b$ for $b \neq 2-d$
(with the understanding that $|x|^0/0 = \log|x|$), one can easily
check that $\Delta W(x) = (2-b -d) |x|^{b-2}$ for all $x \neq 0$,
which leads to a violation of the first line of \eqref{eq:beta-rep} if
$b < 2 -d$. If $b=2-d$, then $\Delta W$ is a multiple of the
Dirac measure at the origin and \eqref{eq:beta-rep} again cannot be
satisfied; the Newtonian potential is therefore not $\beta$-repulsive,
for any $\beta > 0$.  On the opposite, if $2-d < b < 2$, that is, if $W$
has a milder singularity than the Newtonian potential, then it is
$\beta$-repulsive with $\beta = 2 - b$.

Our first assumption on $W$ is the most basic, ensuring that the
interaction energies we use are well defined and have suitable lower
semicontinuity properties:

\begin{hyp}
  \label{hyp:bfb}
  $W\: \R^d \to \R \cup \{+\infty\}$ is lower semicontinuous, bounded
  from below by a finite constant $W_\mathrm{min} \in \R$, and locally
  integrable.
\end{hyp}

In order to prove existence of discrete minimisers we need to add the
following assumption, whose main point is to ensure that $W$ is
attractive at long distances:

\begin{hyp}
  \label{hyp:existence}
  There exists
  $\lim_{|x| \to \infty} W(x) =: W_\infty \in \R \cup \{+\infty\}$,
  $W$ is symmetric, and
  there is $R_W > 0$ such that $W$ is radially strictly increasing on
  $\R^d\setminus B_{R_W}$.
\end{hyp}
As mentioned in the introduction, the condition on the symmetry of $W$
implies no loss of generality since nonsymmetric potentials can be
symmetrised without changing the value of the interaction
energy. Finally, in order to show a uniform bound on the support of
discrete minimisers we need to assume a specific behaviour of the
potential at the origin:

\begin{hyp}
  \label{hyp:W-close-to-0}
  One of the two following properties holds:
  \begin{subhyp}
    \label{hyp:finite}
    \vspace{-0.1cm} $W$ is bounded from
    above and upper semicontinuous.
  \end{subhyp}
  \begin{subhyp}
    \label{hyp:beta}
    \vspace{-0.1cm} $W$ is
    $\beta$-repulsive for some $2 < \beta < d$, it belongs to
    $\C^1(\R^d\setminus \{0\})$, and for some $C_W > 0$ we have
    \begin{align*}
      \Delta^0W(x) \leq C_W
      &\quad \text{for all $x \in \R^d$},
      \\
      W(x) \leq C_W |x|^{2-\beta}
      &\quad \text{for all $|x| \leq 1$},
      \\
      |\nabla W(x)| \leq C_W |x|^{1-\beta}
      &\quad \text{for all $|x| \leq 1$}.
    \end{align*}
  \end{subhyp}
\end{hyp}
Notice that Hypothesis \ref{hyp:finite} and the lower semicontinuity
and boundedness from below of Hypothesis \ref{hyp:bfb} imply that
$W\in\C(\R^d)$. Hypothesis \ref{hyp:finite} tells us that $W(0)$ is
bounded, whereas Hypothesis \ref{hyp:beta} includes unbounded
potentials with a specific repulsive behaviour at the origin. The
radius $1$ in the bounds of $W$ and $|\nabla W|$ is not fundamental
and all proofs work with minor modifications if these bounds hold for
$|x| < r_0$ for a given positive $r_0$. Since we must require that $W$ satisfies Hypotheses
\ref{hyp:bfb}--\ref{hyp:W-close-to-0}, we obtain our results for
singularities up to, and not including, that of the Newtonian
potential $|x|^{2-d}/(d-2)$ (or $-\log |x|$ when $d=2$), with the main
restriction coming from Hypothesis \ref{hyp:beta}.

Typically, the potentials of interest are attractive at long ranges
and repulsive at short ones, and are smooth away from $0$ with a
possible singularity at the origin. As already mentioned, a class of potentials satisfying
Hypotheses \ref{hyp:bfb}--\ref{hyp:W-close-to-0} consists of the power-law
combinations \eqref{eq:power-law}, where we set $W(0) = +\infty$ if $b < 0$. Notice that Hypothesis
\ref{hyp:finite} covers the cases with $b \geq 0$, while Hypothesis
\ref{hyp:beta} covers the cases with $b < 0$.  When $d\in\{1,2\}$ all
power-law potentials of the type \eqref{eq:power-law} fall in the case of
Hypothesis \ref{hyp:finite} due to the condition $0 < b$; in
dimensions $1$ and $2$ the functions $|x|^b$ are not $\beta$-repulsive
(for any $\beta$) if $b \leq 0$.

Let us finally make a note in this section of the terminology used. As it is clear from the introduction, we refer in this paper to \emph{global} minimisers (of the continuum or discrete energy) simply as minimisers. This is because we are not concerned with local minimisers, except on some limited occasions where we clearly mention it as well as the underlying topology in the continuum case. Also, we say that $\rho\in A\subset \P(\R^d)$ is a minimiser of the continuum energy \emph{on} the set $A$ if it minimises the energy among all elements of $A$; this holds in the discrete setting too.

\section{Regularity of continuum minimisers}
\label{sec:prop-continuum}

We make a short observation on the regularity of continuum
minimisers which is essentially contained in the results of
\cite{BCLR13}, but is not mentioned there explicitly. Later, in Section
\ref{sec:morrey-discrete}, we carry out a discrete version of these
arguments. Our main result on continuum minimisers states that they
are bounded in a specific Morrey space of measures for
$\beta$-repulsive potentials. We always denote by $\M(\R^d)$ the space
of finite (signed) Borel measures on $\R^d$.

\begin{defn}[Morrey spaces]\label{defn:morrey-spaces}
  Let $p \in [1,\infty]$ and $\rho \in \M(\R^d)$. We
  say that $\rho$ belongs to the \emph{$p$-Morrey space} $ \M_p(\R^d)$
  if there exists a constant $M > 0$ such that, for all $r > 0$ and
  $x \in \R^d$,
  \begin{equation*}
    |\rho|(B_r(x)) \leq Mr^{d/q},
  \end{equation*}
  where $q$ is the Hölder dual of $p$ and $|\rho|(A)$ is the
  total variation of $\rho$ in a Borel set $A\subset \R^d$. For any
  $\rho \in \M_p(\R^d)$ we define its \emph{$p$-Morrey norm} by
  \begin{equation*}
  \norm{\rho}_{\M_p(\R^d)} = \sup\left\{r^{-d/q}|\rho|(B_r(x)) \st (r,x) \in
    (0,\infty)\times \R^d \right\}.
  \end{equation*}
\end{defn}
Observe that for $p = 1$ we have $q = +\infty$ and the above
definition just states that $\rho$ is finite, so $\M_1(\R^d)$ is just
$\M(\R^d)$ with the total variation norm. Similarly, for $p=\infty$ we
have $q=1$ and $\M_{\infty}(\R^d)$ can be identified with
$L^\infty(\R^d)$.

\begin{thm}\label{thm:morrey}
  Assume $W$ satisfies Hypotheses \ref{hyp:bfb} and
  \ref{hyp:existence} and is unstable. Suppose also that $W$ is
  $\beta$-repulsive for some $0 < \beta < d$ and $\Delta^0W \leq C_W$
  for some $C_W > 0$. If $\rho \in \P(\R^d)$ is a minimiser of the
  continuum interaction energy $E$, then $\rho \in \M_p(\R^d)$ with
  $p = d/(d-\beta)$ and $\|\rho\|_{\M_p(\R^d)} \leq 2^\beta C'$ for
  some $C'>0$ only depending on $W$.
\end{thm}

Note that if $\rho\in \M_p(\R^d)$ then the Hausdorff dimension of the
support of $\rho$ is bounded from below by $d/q$ by Frostman's lemma
\cite{Mattila}; Theorem \ref{thm:morrey} thus tells us that the dimension of the support of a continuum minimiser if at least $\beta$. This dimensionality property is one of the main
results in \cite{BCLR13} and our observation is that almost the same
argument used in \cite{BCLR13} actually reaches the stronger conclusion that
$\rho \in \M_p(\R^d)$. Theorem \ref{thm:morrey} is directly deduced
from the next three lemmas. The first one can be found almost readily in
\cite[Corollary 1]{BCLR13}. The second one states that a minimiser can
be convolved with $\abs{\cdot}^{-\beta}$ to give a bounded function, and
is proved by following and adapting the proof of \cite[Proposition
3]{BCLR13}. The third one comes from potential theory and states that
a probability measure $\rho$ whose convolution with $\abs{\cdot}^{-\beta}$ is bounded
is $p$-Morrey regular for $p = d / (d-\beta)$; it can be found for
example in \cite[Section 8]{Mattila}).

\begin{lem}
  \label{lem:lap0}
  Assume $W$ satisfies Hypotheses \ref{hyp:bfb} and
  \ref{hyp:existence} and is unstable. Suppose also that
  $\Delta^0W\leq C_W$ for some $C_W>0$. If $\rho$ is a minimiser of
  $E$, then $\Delta^0 W *\rho(x)\geq0$ for all $x\in\supp\rho$.
\end{lem}

\begin{proof}
  This is proved in {\cite[Corollary 1]{BCLR13}} with the assumption
  that $W$ is uniformly locally integrable (and not only locally integrable). However, under
  our assumptions, it is proven in \cite{CCP} that all minimisers are
  compactly supported, so that the result holds with the only assumption
  that $W$ is locally integrable.
\end{proof}

\begin{lem} \label{lem:bound-potential-beta} Let $W$ be as in
  Theorem \ref{thm:morrey} and let $\rho \in \P(\R^d)$ be a minimiser
  of the continuum energy. There exists a constant
  $C' > 0$ (depending only on $W$) such
  that \begin{equation*} \ird |x-y|^{-\beta} \d \rho(y) \leq C' \quad
    \mbox{for all $x \in \supp\rho$}.  \end{equation*}
\end{lem}

\begin{proof}
  Choose $x_0 \in \supp\rho$ and write $\rho = \rho_0 + \rho_1$, with
  $\rho_0$ and $\rho_1$ two nonnegative measures such that
  $\supp \rho_0 \subset B_\delta(x_0)$ and
  $\supp \rho_1 \subset \R^d \setminus B_\delta(x_0)$, where $\delta$ is as in Definition \ref{defn:beta-repulsive}, and such that
  neither $\rho_0$ nor $\rho_1$ are zero measures. Now compute
  \begin{align*}
    C \ird |x_0 - y|^{-\beta} \d\rho_0(y)
    &\leq - \ird \Delta^0 W(x_0 - y) \d\rho_0(y)
    \\
    &= -\ird \Delta^0 W(x_0 - y) \d\rho(y)
      + \ird \Delta^0 W(x_0 - y) \d\rho_1(y)
    \\
    &= -\Delta^0 W \ast \rho(x_0) + \ird \Delta^0 W(x_0 - y)
      \d\rho_1(y)
    \\
    &\leq \ird \Delta^0 W(x_0 - y) \d\rho_1(y) \leq C_W,
  \end{align*}
  using the $\beta$-repulsivity of $W$ with $C$ as in Definition \ref{defn:beta-repulsive}, the fact that
  $\Delta^0W \ast \rho(x) \geq 0$ for all $x \in \supp\rho$ by Lemma \ref{lem:lap0}, $\Delta^0W \leq C_W$, and
  $\rho_1(\R^d) \leq 1$. Therefore
  \begin{align*}
    \ird |x_0 - y|^{-\beta} \d\rho(y)
    &\leq \frac{C_W}{C}
      + \ird |x_0 - y|^{-\beta} \d\rho_1(y)\leq 
      \frac{C_W}{C} + \int_{\R^d \setminus B_\delta(x_0)} |x_0 - y|^{-\beta}
      \d\rho_1(y)
    \\
    &\leq \frac{C_W}{C} + \delta^{-\beta} =: C',
  \end{align*}
  using that $\beta > 0$ and
  $\rho_1(\R^d \setminus B_\delta(x_0)) \leq 1$. Notice that the
  constant $C'$ is independent of $x_0$. Thus, since the choice of
  $x_0 \in \supp\rho$ is arbitrary, we get the desired result.
\end{proof}

\begin{lem}\label{lem:morrey}
  Let $0 < \beta < d$ and $\rho \in \P(\R^d)$. Suppose that there
  is a constant $C' > 0$ with \begin{equation*} \ird |x-y|^{-\beta} \d \rho(y) \leq
  C' \quad \mbox{for all $x \in \supp\rho$}.  \end{equation*} Then
  $\rho \in \M_p(\R^d)$ with $p = d/(d-\beta)$ and $\|\rho\|_{\M_p(\R^d)}\leq 2^\beta C'$.
\end{lem}

\begin{proof}
  Let $r > 0$. Then, for all $x \in \supp\rho$,
  \begin{equation*}
    r^{-\beta} \rho(B_r(x)) \leq
    \int_{B_r(x)} |x - y|^{-\beta} \d\rho(y)
    \leq \ird |x-y|^{-\beta} \d \rho(y)
    \leq C',
  \end{equation*}
  since $\beta > 0$. Now suppose that $x \not\in \supp\rho$. Then,
  either $\rho(B_r(x)) = 0$ or $\rho(B_r(x)) > 0$. In the former case,
  we get $r^{-\beta} \rho(B_r(x)) \leq C'$ trivially. In the latter,
  we know that there exists $z \in \supp \rho \cap B_r(x)$ with
  $\rho(B_r(x)) \leq \rho(B_{2r}(z))$. Hence, by the inequality above
  applied to $z$ and $2r$,
  \begin{equation*}
    r^{-\beta} \rho(B_r(x))
    \leq r^{-\beta} \rho(B_{2r}(z))
    \leq 2^\beta (2r)^{-\beta} \rho(B_{2r}(z))
    \leq 2^\beta C'.
  \end{equation*}
  Therefore, writing $M := 2^\beta C'$, we have the result:
  \begin{equation*}
    \rho(B_r(x)) \leq M r^\beta = Mr^{d(1-1/p)} \quad \mbox{for all $x \in \R^d$}. \qedhere
  \end{equation*}
\end{proof}

The previous three lemmas easily imply Theorem \ref{thm:morrey}. For
later use we give the following additional lemma, which is almost a
converse of Lemma \ref{lem:morrey}. It involves a relatively
well-known argument, and can be found for example in \cite[Lemma
2.1]{GG}:

\begin{lem}\label{lem:bound-potential-beta-ball}
  Let $p > 1$, $q = p/(p-1)$ and $0 < \beta < d/q$. For all
  $r > 0$, there exists $C_r > 0$ (depending only on $\beta$, $r$, $q$
  and $d$) such that $C_r\to0$ as $r\to0$ and, for all
  $\rho \in \M_p(\R^d)\cap\P(\R^d)$,
  \begin{equation*}
    \int_{B_r(x)} |x-y|^{-\beta} \d\rho(y) \leq C_r\norm{\rho}_{\M_p(\R^d)}
    \quad \mbox{for all $x \in \R^d$}.
  \end{equation*}
\end{lem}

\begin{proof}
  Let $r > 0$ and $x \in \R^d$, and write
  $D_i(x) := \left\{y \in \R^d \st 2^{-i-1}r \leq |x-y| \leq
    2^{-i}r\right\}$
  for all $i\in\{0,1,2,\dots\}$. Compute
  \begin{align*}
    \int_{B_r(x)} |x-y|^{-\beta} \d\rho(y) &= \sum_{i = 0}^\infty \int_{D_i(x)} |x-y|^{-\beta} \d\rho(y) \leq \sum_{i = 0}^\infty 2^{(i+1)\beta}r^{-\beta} \int_{D_i(x)} \d\rho(y)\\
                                           &\leq \sum_{i = 0}^\infty 2^{(i+1)\beta}r^{-\beta} \int_{\left\{y \in \R^d \st 0 \leq |x-y| \leq 2^{-i}r\right\}} \d\rho(y)\\
                                           & = \sum_{i = 0}^\infty 2^{(i+1)\beta}r^{-\beta} \rho(B_{2^{-i}r}(x))\\
                                           &\leq \sum_{i = 0}^\infty 2^{(i+1)\beta}r^{-\beta} \norm{\rho}_{\M_p(\R^d)} 2^{-id/q} r^{d/q}\\
                                           & = 2^\beta r^{d/q-\beta} \sum_{i = 0}^\infty 2^{i(\beta - d/q)} \norm{\rho}_{\M_p(\R^d)}.
  \end{align*}
  Since $\beta < d/q$ we know
  $\sum_{i = 0}^\infty 2^{i(\beta - d/q)}$ is finite. Setting
  $C_r := 2^\beta r^{d/q-\beta} \sum_{i = 0}^\infty 2^{i(\beta -
    d/q)}$ therefore gives the result.
\end{proof}

Let us remark that Theorem \ref{thm:morrey} actually holds when $\rho\in\P(\R^d)$ has finite energy and it is a local minimiser of the continuum energy with respect to the Wasserstein distance of any finite or infinite order, since Lemma \ref{lem:lap0} stays true in this case; see \cite[Corollary 1]{BCLR13}, and \cite{AGS,Villani2} for an account on transport distances. Wasserstein local minimisers of the continuum energy are therefore Morrey regular under the assumptions on $W$ of Theorem \ref{thm:morrey}.

\section{Properties of discrete minimisers} 
\label{sec:prop-discrete}

\subsection{Existence}
\label{sec:disc-energy}

Let us prove the first part of the main result, Theorem
\ref{thm:main}, regarding the existence of minimisers of the discrete
interaction energy:

\begin{thm}\label{thm:min-Rnd}
  Assume Hypotheses \ref{hyp:bfb} and \ref{hyp:existence}. For any $N \geq 2$ the discrete energy $\E$ has a minimiser on
  $\R^{Nd}$. Furthermore, the diameter of any such minimiser is
  less than $K_N := 2\sqrt{d}(N-1)R_W$ (which only depends on $N$ and $W$).
\end{thm}

Theorem \ref{thm:min-Rnd} is proved by considering minimisers in
$(\overline{B}_R)^N$ for some $R\geq 0$, and by showing a uniform bound on their diameter,
independently of $R$. This is stated in the following lemma:

\begin{lem}
  \label{lem:min-BRN}
  Suppose that $W$ satisfies Hypotheses \ref{hyp:bfb} and
  \ref{hyp:existence}, and let $R\geq0$. There exists a
  minimiser of $\E$ on $(\overline{B}_R)^N$. If $\bX$ is any such
  minimiser, then
  \begin{equation*}
    \diam\bX \leq 2 \sqrt{d} (N-1)R_W =: K_N.
  \end{equation*}
\end{lem}

Observe that our control of the support of the minimiser given by
Lemma \ref{lem:min-BRN} depends on $N$. This is an easy estimate which
holds under weak conditions on $W$; later, in Theorem
\ref{thm:discrete-diameter}, we show that in fact, when $W$ is
unstable, the size of the support of $N$-particle minimisers stays
uniformly bounded in $N$, and that constitutes one of the central
arguments in this paper.

\begin{proof}[Proof of Lemma \ref{lem:min-BRN}]
  The fact that a minimiser exists is straightforward by compactness
  of $(\overline{B}_R)^N$ and lower semicontinuity of $\E$ (since $W$
  is lower semicontinuous). Let then $\bX$ be a minimiser of
  $\E$ on $(\overline{B}_R)^N$.
	
  Denote by $\pi_k\:\R^d\to\R$ the projection on the $k$th axis. We
  want to prove the following claim. In each coordinate there cannot
  be ``gaps'' greater that $2R_W$ among any particles of $\bX$: if
  $k\in\{1,\dots,d\}$ and $a_k\in\R$ is so that
  $x_{i} \not\in \pi_k^{-1}([a_k-R_W,a_k+R_W])$ for all $i\in\oneN$,
  then either $x_{i}\not\in\pi_k^{-1}((-\infty,a_k-R_W])$ for all
  $i\in\oneN$ or $x_{i}\not\in\pi_k^{-1}([a_k+R_W,\infty])$ for all
  $i\in\oneN$. Without loss of generality we prove the claim for
  $k=1$. We proceed by contradiction: assume that there is $a_1\in\R$
  such that
  $I_\mt{L} := \{i\in\oneN \mid x_i \in
  \pi_1^{-1}((-\infty,a_1-R_W])\} \neq \emptyset$,
  $I_\mt{R} := \{i\in\oneN \mid x_i \in \pi_1^{-1}([a_1+R_W,\infty))\}
  \neq \emptyset$
  and $\oneN \setminus (I_\mt{L} \cup I_\mt{R}) = \emptyset$. By
  renaming the particles we may assume that
  $I_\mt{L}=\{1,\dots,N_\mt{L}\}$ and
  $I_\mt{R}=\{N_\mt{L}+1,\dots,N\}$ for some $1\leq N_\mt{L} < N$.  Let
  $0<\e_1\leq R_W$ and $\e=(\e_1,0,\dots,0) \in \R^d$, and define the
  ``left-shifted'' particles \begin{equation*} \boldsymbol{X'} =
    (x_1',\dots,x_N') :=
    (x_{1},\dots,x_{N_\mt{L}},x_{N_\mt{L}+1}-\e,\dots,x_{N}-\e) \in
    (\overline{B}_R)^N.  \end{equation*} Let us compute the discrete
  energy of $\boldsymbol{X'}$.
  \begin{align*}
    N^2\E(\boldsymbol{X'})
    &= \frac12 \sum_{i\in I_\mt{L}} \sum_{\substack{j\in I_\mt{L}\\j\neq i}} W(x_i'-x_j') + \frac12 \sum_{i\in I_\mt{R}} \sum_{\substack{j\in I_\mt{R}\\j\neq i}} W(x_i'-x_j') + \sum_{i\in I_\mt{L}} \sum_{j\in I_\mt{R}} W(x_i'-x_j')\\
    &= \frac12 \sum_{i\in I_\mt{L}} \sum_{\substack{j\in I_\mt{L}\\j\neq i}} W(x_{i}-x_{j}) + \frac12 \sum_{i\in I_\mt{R}} \sum_{\substack{j\in I_\mt{R}\\j\neq i}} W(x_{i} -\e - (x_{j}-\e))\\
    &\phantom{{}={}} + \sum_{i\in I_\mt{L}} \sum_{j\in I_\mt{R}} W(x_{i}-x_{j} + \e).
  \end{align*}
  Let $x_{i,1}:= \pi_1(x_{i})$ for any $i\in\oneN$. Since clearly
  $x_{i,1} - x_{j,1} + \e_1 \leq -2R_W + \e_1\leq -R_W$ for all
  $(i,j)\in I_\mt{L}\times I_\mt{R}$, Hypothesis \ref{hyp:existence} gives
  \begin{align*}
    N^2\E(\boldsymbol{X'})
    &< \frac12 \sum_{i\in I_\mt{L}} \sum_{\substack{j\in I_\mt{L}\\j\neq i}} W(x_{i}-x_{j}) + \frac12 \sum_{i\in I_\mt{R}} \sum_{\substack{j\in I_\mt{R}\\j\neq i}} W(x_{i} - x_{j}) + \sum_{i\in I_\mt{L}} \sum_{j\in I_\mt{R}} W(x_{i}-x_{j})\\
    &= N^2\E(\bX),
\end{align*}
which is a contradiction of $\bX$ being a minimiser on
$(\overline{B}_R)^N$, which shows the claim.

To complete the proof of the lemma note that the above claim implies
that the diameter of the $k$th projection of the set
$\{x_1, \dots, x_N\}$ is less than $2(N-1)R_W$. Since this is true of
all projections, we deduce that
$\diam\{x_1, \dots, x_N\} \leq 2 \sqrt{d} (N-1) R_W$, which ends the
proof.
\end{proof}

We can now prove Theorem \ref{thm:min-Rnd}.
\begin{proof}[Proof of Theorem \ref{thm:min-Rnd}]
  By Lemma \ref{lem:min-BRN} we know that there is a minimiser of
  $\E$ on $(\overline{B}_{K_N})^N$, say $\bX$. We want to prove that
  $\bX$ is actually a minimiser on all of $\R^{Nd}$. Let
  $\boldsymbol{X'}\in\R^{Nd}$. Necessarily, there exists $R\geq0$ such
  that $\boldsymbol{X'}\in(\overline{B}_R)^N$. By Lemma
  \ref{lem:min-BRN} take a minimiser on $(\overline{B}_R)^N$,
  say $\bY$. We know that the diameter of $\bY$
  is less than or equal to $K_N$, so by possibly translating $\bY$
  (and by translation invariance of $\E$) we may assume that
  $\bY \in (\overline{B}_{K_N})^N$ without loss of
  generality. Therefore
  $\E(\bX) \leq \E(\bY) \leq E_N(\boldsymbol{X'})$, which shows, by the
  arbitrariness of the choice of $\boldsymbol{X'}$, that $\bX$ is a
  minimiser of $\E$. This proves the first part of Theorem
  \ref{thm:min-Rnd}. The second part is straightforward: if
  $\bX\in\R^{Nd}$ is a minimiser of $\E$, then
  $\bX\in(\overline{B}_R)^N$ for some $R\geq0$, and so, by Lemma
  \ref{lem:min-BRN}, its diameter is less than or equal to $K_N$.
\end{proof}

\subsection{Morrey-type regularity}
\label{sec:morrey-discrete}

This section is the discrete analogue of Section
\ref{sec:prop-continuum}. As explained in the introduction, we define
a discrete counterpart of the classical Morrey spaces of Definition
\ref{defn:morrey-spaces}.

\begin{defn}[Empirical Morrey measures]\label{defn:discrete-morrey-spaces}
  Let $p\in[1,\infty]$ and $\bX \in \R^{Nd}$. We say that $\mu_{\bX}$ is an
  \emph{empirical (or discrete) $p$-Morrey measure} if there exists
  $M > 0$ such that, for all $r > 0$ and $i\in\{1,\dots,N\}$,
  \begin{equation}
    \label{eq:notm}
    m_{i,r}(\bX)  := \mu_{\bX} \big( B_r(x_i) \big) - \tfrac1N
    \leq Mr^{d/q}, 
  \end{equation}
  where $q$ is the Hölder dual of $p$. In this case we write
  $\mu_{\bX} \in \M_p^N$, or simply $\bX\in \M_p^N$. We also write
  \begin{equation*}
    [\mu_{\bX}]_{\M_p^N} = [\bX]_{\M_p^N} : = \sup\left\{r^{-d/q}m_{i,r}(\bX) \mid (r,i) \in (0,\infty)\times\{1,\dots,N\}\right\}.
  \end{equation*}
\end{defn}

Given a configuration $\bX = (x_1, \dots, x_N) \in \R^{Nd}$,
throughout this paper we denote by $m_{i,r}(\bX)$ the total mass in
the open ball of radius $r$ centred at $x_i$, not counting the $i$th
particle, as defined in \eqref{eq:notm}. Note that, unlike
$\|\cdot\|_{\M_p(\R^d)}$, $[\,\cdot\,]_{\M_p^N}$ does not define a norm;
$\M_p^N$ is not a Banach space or even a linear vector space.

\medskip
We prove the following discrete regularity, an analogue of Theorem
\ref{thm:morrey}:

\begin{thm}
  \label{thm:discrete-morrey}
  Suppose that $W$ satisfies Hypothesis \ref{hyp:bfb}, it is $\beta$-repulsive with $0 < \beta<d$ and
  $\Delta^0 W \leq C_W$ for some $C_W >0$. If $\bX \in \R^{Nd}$ is a
  minimiser of the discrete interaction energy $\E$, then $\bX \in \M_p^N$ with
  $p = d/(d-\beta)$ and $[\bX]_{\M_p^N} \leq C'$ for some $C' > 0$ only depending on $W$.
\end{thm}

The proof of Theorem \ref{thm:discrete-morrey} consists of the
following three lemmas. The reader can follow the parallel with
Section \ref{sec:prop-continuum}. The following notation is used
throughout this paper: for any $r>0$, $\bX\in\R^{Nd}$ and $i\in\oneN$,
we write $S_{i,r}(\bX)$ to denote the set of indices of particles different
from $i$ which are at distance less than $r$ from $x_i$, i.e.,
\begin{equation}
  \label{eq:notS}
  S_{i,r}(\bX) := \{j\in\oneN \mid j\neq i,
  |x_i - x_j| < r\}
\end{equation}
and by $T_{i,r}(\bX)$ its complement, still removing $i$, that is,
\begin{equation}
  \label{eq:notT} T_{i,r}(\bX) := \{j\in\oneN \mid\
  |x_i - x_j| \geq r\}.
\end{equation}

\begin{lem}
  \label{lem:Lap0}
  Assume that $W$ satisfies Hypothesis \ref{hyp:bfb}, and let
  $\bX \in \R^{Nd}$ be a minimiser of $E_N$. Then
  \begin{equation} \label{eq:Lap0}
    \sum_{\substack{i=1\\i \neq j}}^N \Delta^0 W(x_i - x_j)
    \geq 0
    \quad \text{ for all $j \in\{1,\dots,N\}$.}
  \end{equation}
\end{lem}
\begin{proof}
	We write the minimiser $\bX = (x_1,\dots,x_N)$. For all $j\in\oneN$ define
\begin{equation*}
	p_j(x) := \frac1N \sum_{\substack{i=1\\i\neq j}}^N W(x_i-x) \quad \mbox{for all $x\in\R^d$}.
\end{equation*}
Consider $f_1\: x\mapsto N\E(x,x_2,\dots,x_N)$ and compute
\begin{equation*}
	f_1(x) = \frac1N\sum_{i=2}^N W(x_i-x) + \frac{1}{2N}\sum_{j=2}^N \sum_{\substack{i=2\\i\neq j}}^N W(x_i-x_j) = p_1(x) + \frac{1}{2N}\sum_{i=2}^N \sum_{\substack{i=2\\i\neq j}}^N W(x_i-x_j).
\end{equation*}
By the optimality of $\bX$ we know that $x_1$ is a minimiser of
$f_1$ on $\R^d$. The very last term of the above computation is
independent of $x$ and therefore $x_1$ is also a minimiser of
$p_1$ on $\R^d$. Hence $\Delta^0 p_1(x_1) \geq 0$. By repeating the
above argument for all $j\geq2$ we finally get
$\Delta^0 p_j(x_j)\geq0$ for all $j\in\oneN$, which is the result.
\end{proof}

\begin{lem}
  \label{lem:beta-power}
  Let $W$ be as in Theorem \ref{thm:discrete-morrey} and let
  $\bX \in \R^{Nd}$ be a minimiser of the discrete energy. There exists a
  constant $C' > 0$ (depending only on $W$) such that
  \begin{equation*}
    \frac1N \sum_{\substack{i=1\\i \neq j}}^N |x_i - x_j|^{-\beta}
    \leq C' \quad \text{ for all $j \in \{1,\dots,N\}$.}
  \end{equation*}
\end{lem}

\begin{proof}
  We prove it for $j=1$ without loss of generality. Let $\delta$ and $C$ be
  the constants appearing in the definition of $\beta$-repulsivity. In
  \eqref{eq:Lap0} we can separate the terms where the singularity of
  $\Delta^0 W$ is important to obtain
  \begin{align*}
    0 \leq \sum_{i=2}^N \Delta^0 W(x_i - x_1)
    &\leq - C \sum_{i\in S_{1,\delta}(\bX)}
      |x_i - x_1|^{-\beta}
      + \sum_{i\in T_{1,\delta}(\bX)} \Delta^0 W(x_i -x_1)
    \\
    &\leq - C\sum_{i\in S_{1,\delta}(\bX)} |x_i - x_1|^{-\beta} + C_WN,
  \end{align*}
  with the notation given in \eqref{eq:notS} and \eqref{eq:notT}. This implies that
  $\sum_{i\in S_{1,\delta}(\bX)} |x_i - x_1|^{-\beta} \leq (C_W/C)N$, and consequently
  \begin{equation*}
    \frac{1}{N} \sum_{i = 2}^N  |x_i - x_1|^{-\beta}  = \frac{1}{N} \sum_{i\in S_{1,\delta}(\bX)}  |x_i - x_1|^{-\beta} + \frac{1}{N} \sum_{i\in T_{1,\delta}(\bX)} |x_i - x_1|^{-\beta} \leq \frac{C_W}{C} + \delta^{-\beta},
  \end{equation*}
  which yields the result with $C' := C_W/C + \delta^{-\beta}$.
\end{proof}

\begin{lem}
  \label{lem:beta-implies-Morrey}
  Let $0<\beta<d$ and $\bX \in \R^{Nd}$. Suppose that there is a constant $C'>0$ with
  \bes
    \frac1N \sum_{\substack{i=1\\i \neq j}}^N |x_i - x_j|^{-\beta}
    \leq C' \quad \text{ for all $j \in \{1, \dots, N\}$}
  \ees
  Then $\bX\in\M_p^N$ for
  $p = d/(d-\beta)$ and $[\bX]_{\M_p^N} \leq C'$.
\end{lem}

\begin{proof}
  We want to prove that $m_{j,r}(\bX) \leq C' r^\beta$ for all $r>0$ and $j\in\{1,\dots,N\}$,
  with the notation in \eqref{eq:notm}. Without loss of generality,
  assume $j=1$. We have
  \begin{equation*}
    r^{-\beta} m_{1,r}(\bX) \leq \frac1N \sum_{i \in S_{1,r}(\bX)} |x_i - x_1|^{-\beta}
    \leq \frac1N\sum_{i=2}^N |x_i - x_1|^{-\beta} \leq C',
  \end{equation*}
  since $\beta>0$, which is the result.
\end{proof}

We give an additional lemma whose proof is analogous to that of Lemma
\ref{lem:bound-potential-beta-ball} and we omit:

\begin{lem}
  \label{lem:bound-potential-beta-ball-discrete}
  Let $p > 1$, $q = p/(p-1)$ and $0 < \beta < d/q$. For all $r > 0$,
  there exists $C_r > 0$ (depending only on $\beta$, $r$, $q$ and $d$)
  such that $C_r\to0$ as $r \to 0$ and, for all $\bX \in \M_p^N$,
  \begin{equation*}
    \frac{1}{N} \sum_{i\in S_{j,r}} |x_i-x_j|^{-\beta}
    \leq C_r [\bX]_{\M_p^N} \quad \mbox{for all $j\in\oneN$},
  \end{equation*}
  where we refer the reader to the notation in \eqref{eq:notS}.
\end{lem}

Observe that, as for the continuum case in Section \ref{sec:prop-continuum}, Theorem \ref{thm:discrete-morrey} actually holds when $\bX\in\R^{Nd}$ has finite energy and it is a local minimiser of the discrete energy since one can easily check that Lemma \ref{lem:Lap0} stays true in this case. Local minimisers of the discrete energy are therefore discretely Morrey regular under the assumptions on $W$ of Theorem \ref{thm:discrete-morrey}.

\subsection{Euler--Lagrange estimate}
\label{subsec:eulerlagrange}

We prove an Euler--Lagrange estimate at the discrete level, as
discussed in the introduction. It is the discrete analogue of the
Euler--Lagrange equation given in the first line of
\eqref{eq:ELcont}. Recall that for every $\bX\in\R^{Nd}$ we write
\begin{equation*}
  P_i(\bX) := \frac1N\sum_{\substack{j=1\\j\neq i}}^N W(x_i-x_j)
  \quad \mbox{for all $i\in\oneN$}.
\end{equation*}
Note that $P_i(\bX) = p_i(x_i)$, where $p_i$ already appeared in the
proof of Lemma \ref{lem:Lap0}.

\begin{thm}
  \label{thm:constant-bounded}
  Suppose that $W$ satisfies Hypothesis \ref{hyp:bfb}--\ref{hyp:W-close-to-0} and let $\bX \in \R^{Nd}$ be a minimiser of $\E$. If $W$ satisfies Hypothesis \ref{hyp:finite}, then
  \be\label{eq:potential-energy}
    |P_i(\bX) - 2\E(\bX)| \leq \frac{W(0)-W_\mt{min}}{N}
    \quad \mbox{for all $i\in\oneN$}.
  \ee
If $W$ satisfies Hypothesis \ref{hyp:beta}, then there exist $A > 1$ and $k \in (0,1]$
  (independent of $N$ and $\bX$) such that
  \begin{equation*}
    |P_i(\bX) - 2\E(\bX)| \leq AN^{-k}
    \quad \text{for  all $i \in \oneN$}.
  \end{equation*}
  One can take $k = 2/((\beta-1)\beta)$.
\end{thm}

\begin{proof}
  First suppose that $W$ satisfies Hypothesis \ref{hyp:finite}. We first prove
  \begin{equation} \label{eq:Pij}
    |P_i(\bX) - P_j(\bX)| \leq \frac{W(0)-W_\mt{min}}{N}
    \quad \mbox{for all $i,j\in\oneN$}.
  \end{equation}
  To this end we proceed by contradiction by assuming that the result
  is not true. We move one particle of the minimiser at the exact
  location of another particle of the minimiser
  and show that the resulting energy is lower. With no loss of generality, suppose that
  $N(P_1(\bX) -P_2(\bX)) > W(0)-W_\mt{min}$, and that we move $x_1$ at the
  location of $x_2$. Write $\boldsymbol{X'}:= (x_2,x_2,x_3,\dots,x_N)$
  and compute
  \begin{align*}
    N^2\left(\E(\bX) - \E(\boldsymbol{X'})\right)
    &= \sum_{i=2}^N W(x_1 - x_i) - \sum_{i=2}^N W(x_2 - x_i)
    \\
    &= N(P_1(\bX) - P_2(\bX)) - W(0) + W(x_1-x_2)
    \\
    &> W(0) -W_\mt{min} -W(0) + W_\mt{min} = 0,
  \end{align*}
  which is a contradiction to the fact that $\bX$ is a
  minimiser of $\E$, and shows \eqref{eq:Pij}.
  Averaging \eqref{eq:Pij} over $j = 1, \dots, N$ gives, for all $i\in\{1,\dots,N\}$,
  \begin{align*}
    \frac{W(0)-W_\mt{min}}{N}
    &\geq \frac1N \sum_{j=1}^N |P_i(\bX) - P_j(\bX)|
    \\
    &\geq \left| P_i(\bX) - \frac1N \sum_{j=1}^N P_j(\bX) \right|
      = \left| P_i(\bX) - 2\E(\bX)\right|,
  \end{align*}
  which shows \eqref{eq:potential-energy}.

Suppose now that $W$ satisfies Hypothesis \ref{hyp:beta}. We first want to prove that, for some $A>1$ and $k\in(0,1]$,
  \begin{equation} \label{eq:Pij2}
    |P_i(\bX) - P_j(\bX)| \leq AN^{-k} \quad \mbox{for all $i,j\in\oneN$}.
  \end{equation}
  To this end we intend to reach a contradiction by assuming that
  \begin{equation*}
    P_1(\bX) - P_2(\bX) > AN^{-k}
  \end{equation*}
  for some $N\geq2$ arbitrarily large, and for some $A > 1$ and $k \in (0,1]$ to be chosen appropriately later. We intend to reach a contradiction for certain
  values of $A$ and $k$. We move the first particle of the
  minimiser, located at $x_1$, to a point $x_2'$ close to the second
  particle, located at $x_2$. We want to show that, with an
  appropriate choice of $x_2'$, the resulting energy is lower. Write
  $\boldsymbol{X'}:= (x_2',x_2,x_3,\dots,x_N)$ and compute
  \begin{align*}
    \label{eq:4}
    \numberthis 2N^2\left(\E(\bX) - \E(\boldsymbol{X'})\right)
    &= \sum_{i=2}^N W(x_1 - x_i) - \sum_{i=2}^N W(x_2' - x_i)
    \\
    &= N(P_1(\bX) - P_2(\bX))\\
    &\phantom{{}={}}+ \sum_{i=3}^N \left( W(x_2-x_i) - W(x_2'-x_i)
    \right)
    \\
    &\phantom{{}={}}- W(x_2'-x_2) + W(x_2-x_1)
    \\
    &> A N^{1-k} + \sum_{i=3}^N \left( W(x_2-x_i) - W(x_2'-x_i)
    \right)\\
    &\phantom{{}={}}- W(x_2'-x_2) + W_\mathrm{min}.
  \end{align*}
  We need to bound from below the remaining terms involving $W$ and
  show that they are strictly greater than $-A N^{1-k}$. To this end,
  $x_2'$ needs to be chosen carefully. We know by Lemma
  \ref{lem:beta-power} that
  \begin{equation} \label{eq:estimate}
    \frac1N \sum_{\substack{i=1\\i \neq 2}}^N |x_i - x_2|^{-\beta} \leq C',
  \end{equation}and in particular
  \begin{equation*}
    |x_i - x_2| \geq (C' N)^{-1/\beta} = C_1 N^{-1/\beta} \quad \text{ for all $i \neq 2$},
  \end{equation*}
  where $C_1 := (C')^{-1/\beta}$. Thus there are no other particles in a
  radius $C_1 N^{-1/\beta}$ around $x_2$. We pick $x_2'$ at less than
  half that distance to make sure that we stay away from other
  particles: we take
  \begin{equation} \label{eq:alpha-cond} 
    \alpha \geq 1/\beta, 
  \end{equation}
  to be chosen later, and pick $x_2'$ so that
  \begin{equation}
    \label{eq:13}
    2|x_2' - x_2| = C_1 N^{-\alpha} \leq C_1 N^{-1/\beta} \leq
    |x_i-x_2|
    \quad \text{for all $i \neq 2$.}
  \end{equation}
  Let us then bound the terms in \eqref{eq:4} directly involving $W$. We have, by Hypothesis \ref{hyp:beta} and since $N$ is large enough so that $C_1N^{-1/\beta}/2\leq 1$,
  \begin{equation*}
    W(x_2' - x_2) \leq C_W |x_2' - x_2|^{2-\beta}
    = C_W \left(\frac{C_1N^{-\alpha}}{2} \right)^{2-\beta}
    = C_2 N^{\alpha (\beta-2)},
  \end{equation*}
  where $C_2:= C_W(C_1/2)^{2-\beta}$. Since we need this to be
  smaller than $AN^{1-k}$, we impose
  $k = 1-\alpha (\beta-2)$, so that
  \begin{equation} \label{eq:est1}
    W(x_2' - x_2) \leq C_2 N^{1-k}.
  \end{equation}
  For the other term, pick a cut-off distance
  $\ell = \ell(N) < 1/3$, to be chosen later, and let
  \begin{equation*}
    S := S_{2,\ell}(\bX) \setminus \{1\},
  \end{equation*}
  where we refer the reader to the notation given in \eqref{eq:notS}.
  We write
  \begin{align*}\label{eq:9}
    \numberthis \left|
    \sum_{i=3}^N \left( W(x_2-x_i) - W(x_2'-x_i) \right)
    \right|&
    \leq \sum_{ \substack{i=3\\i \notin S} }^N
    \left| W(x_2-x_i) - W(x_2'-x_i) \right|
    \\
    &+ \sum_{ i \in S } \left| W(x_2-x_i) \right|
      + \sum_{ i \in S } \left| W(x_2'-x_i) \right|.
  \end{align*}
  The next-to-last term can be estimated, using
  \eqref{eq:estimate} and Hypothesis \ref{hyp:beta}, as
  \begin{align*}
    \frac1N \sum_{ i \in S }
    \left| W(x_2-x_i) \right|
    &\leq \frac{C_W}{N} \sum_{ i \in S } | x_2-x_i |^{2-\beta}\\
    &\leq C_W \left(
      \frac{1}{N} \sum_{ i \in S } | x_2-x_i |^{-\beta}
      \right)^{(\beta-2)/\beta}
      \left(\frac{|S|}{N}\right)^{2/\beta}
    \\ 
    &\leq C_W (C')^{(\beta-2)/\beta}
      \left(\frac{|S|}{N}\right)^{2/\beta}
      = C_3 \left(\frac{|S|}{N}\right)^{2/\beta},
  \end{align*}
  where $C_3:= C_W (C')^{(\beta-2)/\beta}$. On the other hand, due to Lemma
  \ref{lem:beta-implies-Morrey}, we have
  \begin{equation*}
    \frac{|S|}{N} \leq C'\ell^{\beta}.
  \end{equation*}
  Hence
  \begin{equation*}
    \sum_{ i \in S } \left| W(x_2-x_i) \right|
    \leq C_3 (C')^{2/\beta}  \ell^{2} N = C_4 \ell^{2} N,
  \end{equation*}
  where $C_4:= C_3 (C')^{2/\beta}$. This motivates the choice
  $\ell := N^{-k/2}$ which is less than $1/3$ for $N$ large enough, so that
  \begin{equation} \label{eq:est3}
    \sum_{ i \in S } \left| W(x_2-x_i) \right| \leq C_4 N^{1-k}.
  \end{equation}
  The last term in \eqref{eq:9} is comparable to the one we just
  bounded, since $|x_i - x_2|$ and $|x_i - x_2'|$ are comparable
  due to \eqref{eq:13}. Indeed,
  \begin{equation*}
    |x_i - x_2|
    \leq |x_i - x_2'| + |x_2' - x_2| = |x_i - x_2'|
    + \textstyle\frac12 |x_i - x_2|,
  \end{equation*}
  so that
  \begin{equation*}
    |x_i - x_2| \leq 2 |x_i - x_2'|.
  \end{equation*}
  With this, and what we proved above,
  \begin{equation}
    \label{eq:est4}
    \sum_{ i \in S } \left| W(x_2'-x_i) \right|
    \leq C_W \sum_{ i \in S } |x_2'-x_i|^{2-\beta}
    \leq 2^{\beta-2} C_W \sum_{ i \in S } |x_2-x_i|^{2-\beta}
    \leq C_5 N^{1-k},
  \end{equation}
  where $C_5:= 2^{\beta-2} C_WC_4$. Finally, for the first term in
  \eqref{eq:9}, notice that for $i \in S$ we have $|x_2 - x_i|
  \leq \ell$ (by definition of $S$), and also
  \begin{equation*}
    |x_2' - x_i| \leq |x_2'-x_2| + |x_2-x_i| \leq
    \textstyle \frac12 |x_2-x_i| + |x_2-x_i| \leq \textstyle \frac32 \ell
  \end{equation*}
  due to \eqref{eq:13}. Since we are requiring $\ell < 1/3$, both
  $x_2-x_i$ and $x_2'-x_i$ are in the ball of radius 1 centred at
  0 and we may use the gradient bound in Hypothesis
  \ref{hyp:beta} to get
  \begin{align*}
    \label{eq:est5}
    \numberthis \sum_{ \substack{i=3\\i \notin S} }^N
    \left| W(x_2-x_i) - W(x_2'-x_i) \right|
    &\leq C_W N \ell^{1-\beta} |x_2-x_2'|\\
    &\leq \textstyle \frac{C_1}{2} C_WN \ell^{1-\beta} N^{-\alpha}
    = C_6 N^{1 + k(\beta-1)/2 - \alpha}, 
  \end{align*}
  where $C_6:= C_1C_W/2$, thanks to \eqref{eq:13}. Since $\alpha (\beta-2) = 1-k$, choose
  $k$ so that $1 + k (\beta-1)/2 - (1-k)/(\beta-2) = 1-k$, that is,
  \begin{equation*}
    k := \frac{2}{(\beta-1)\beta}.
  \end{equation*}
  This gives
  \begin{equation*}
    \alpha = \frac{1-k}{\beta-2} = \frac{\beta+1}{(\beta-1)\beta}\geq \frac1\beta,
  \end{equation*}
  as required in
  \eqref{eq:alpha-cond}. Putting together \eqref{eq:4},
  \eqref{eq:est1}, \eqref{eq:9}, \eqref{eq:est3}, \eqref{eq:est4}
  and \eqref{eq:est5},
  \begin{equation*}
    2N^2\left(\E(\bX) - \E(\boldsymbol{X'})\right)
    > (A - C_2 - C_4 - C_5 - C_6) N^{1-k} + W_\mathrm{min}.
  \end{equation*}
  An appropriate choice of the constant $A$ makes this quantity
  positive for all $N$ large enough, contradicting the fact that $\bX$
  is a minimiser of $\E$, thus showing \eqref{eq:Pij2}.

  Averaging \eqref{eq:Pij2} over $j = 1, \dots, N$ we get, for all
  $i\in\oneN$,
  \begin{equation*}
    \frac{A}{N^{k}}
    \geq \frac1N \sum_{j=1}^N |P_i(\bX) - P_j(\bX)|
    \geq \left| P_i(\bX) - \frac1N \sum_{j=1}^N P_j(\bX) \right|
    = \left| P_i(\bX) - 2\E(\bX)\right|,
  \end{equation*}
  which ends the proof.
\end{proof}

\subsection{Diameter estimates}
\label{sec:diameter-discrete}

As a tool to prove our main result we need to introduce the following
notion of discrete instability:

\begin{defn}[Discrete instability]
  \label{defn:discrete-instability}
  Let $W\:\R^d\to\R\cup\{+\infty\}$ and suppose that $W_\infty: =\lim_{|x|\to\infty} W(x)$ exists (possibly
  $+\infty$). We say that $W$ is \emph{discretely unstable (with
    constant $s > 0$)} if there exist $s>0$ and $\bar N\geq2$ such
  that, for all $N>\bar N$, there exists $\bX\in\R^{Nd}$
  with
  \begin{equation*}
    E_N(\bX) < \textstyle{\frac12 W_\infty-s}.
  \end{equation*}
\end{defn}
Note that if $W$ is discretely unstable with some constant $s$, then
it is so with any $s' < s$. This definition is a
natural discrete version of the instability in Definition
\ref{defn:instability}, and it is the one we need in order to
carry out the next arguments. Actually, both concepts
turn out to be equivalent under Hypotheses
\ref{hyp:bfb}--\ref{hyp:W-close-to-0}; see Proposition
\ref{prop:unstable}. 

\begin{lem}
  \label{lem:lower-mass}
  Assume that $W$ satisfies Hypotheses
  \ref{hyp:bfb}--\ref{hyp:W-close-to-0} and is discretely
  unstable. There exist $\bar N\geq2$ and $r, m > 0$ depending only on
  $W$ such that, for each $N>\bar N$ and any minimiser $\bX$ of $E_N$
  on $\R^{Nd}$ it holds that
  \begin{equation*}
    m_{i,r}(\bX) \geq m
    \quad \mbox{for all $i\in\oneN$},
  \end{equation*}
  where we use the notation in \eqref{eq:notm}.
\end{lem}

\begin{proof}
  Suppose first that $W$ satisfies Hypothesis
  \ref{hyp:finite}. Let $\bX$ be a minimisers of $\E$ and write $E_N^0 := \E(\bX)$. Then, by Theorem \ref{thm:constant-bounded}, for all
  $i\in\oneN$, \begin{equation*} P_i(\bX) \leq 2\E^0+ \frac{W(0) -
      W_\mt{min}}{N}.  \end{equation*} Let $s>0$ be the constant in the definition of discrete instability. We can pick
  $\bar N\geq2$ such that, for all $N>\bar N$,
  $E_N^0 < W_\infty/2 - s$. Thus, 
  \begin{equation*}
    E_0:= \sup_{N\geq \bar N} E_N^0 \leq \textstyle \frac 12 W_\infty
    -s < \textstyle\frac12 W_\infty.
  \end{equation*} 
  Let $\rho\in\P(\R^d)$ be such that $E(\rho)<+\infty$, which exists by local integrability of $W$. By Lemma \ref{lem:discrete-approximation} there exists a sequence of particle configurations $(\bXNs)_{N\geq2}$ such that
  \begin{equation*} 
  \limsup_{N\to\infty}
  E_N^0 \leq \lim_{N\to\infty} \E(\bXNs) = E(\rho) < +\infty,
  \end{equation*} 
  so that $E_0$ is finite even if $W_\infty$ is not. We can then
  take $a$ such that
  $W_\mt{min}/2 \leq E_0 < a < W_\infty/2$. Let $r>0$
  be such that $W(x) > 2a$ for all $|x|>r$. Compute, for all
  $N>\bar N$,
\begin{align*}
	2E_0 + \frac{W(0) - W_\mt{min}}{N} \geq P_i(\bX) &= \frac1N\sum_{j\in S_{i,r}(\bX)} W(x_i-x_j) + \frac1N\sum_{j\in T_{i,r}(\bX)} W(x_i-x_j)\\
	&\geq W_\mt{min} m_{i,r}(\bX) + \frac{2a}{N} \sum_{j\in T_{i,r}(\bX)} 1\\
	&= W_\mt{min} m_{i,r}(\bX) + 2a \left(1- m_{i,r}(\bX) -\frac1N\right)\\
	&= (W_\mt{min} - 2a) m_{i,r}(\bX) + 2a\left(1-\frac1N\right),
\end{align*}
where the notation is as in \eqref{eq:notS} and \eqref{eq:notT}. Since $W_\mt{min} < 2a$ we get
\begin{equation*}
	m_{i,r}(\bX) \geq \frac{2E_0 -2a + N^{-1}(W(0) - W_\mt{min} +2a)}{W_\mt{min} -2a}.
\end{equation*}
Since $E_0 < a$, there exists $b >0$ such that $b< 2a - 2E_0$. Then,
for some number of particles large enough, which we still denote by $\bar N$, we have
$(W(0) - W_\mt{min} +2a)/N < 2a - 2E_0 - b$ for all $N > \bar
N$. Therefore,
\begin{equation*}
  m_{i,r}(\bX) \geq \frac{-b}{W_\mt{min} -2a} =: m>0 \quad \mbox{for all $N>\bar N$.}
\end{equation*}
The choices of $a$ and $b$ only depend on $W$ and therefore $r$ and $m$ only depend on $W$ as well, which shows the result when Hypothesis \ref{hyp:finite} holds.

If now $W$ satisfies Hypothesis \ref{hyp:beta}, then the
arguments above can still be carried out in the same fashion using the second part of
Theorem \ref{thm:constant-bounded} instead of the first.
\end{proof} 

\begin{thm}
  \label{thm:discrete-diameter}
  Assume that $W$ satisfies Hypotheses
  \ref{hyp:bfb}--\ref{hyp:W-close-to-0} and it is discretely unstable. There is a
  constant $K > 0$ depending only on $W$ (in particular, independent
  of $N$) such that the diameter of any discrete minimiser is less than $K$.
\end{thm}

\begin{proof}
  Let $\bar N\geq2$, $m$ and $r$ be as in Lemma \ref{lem:lower-mass},
  and let $\bX \in \R^{Nd}$ be a minimiser of $E_N$ for some $N > N$.
  We can carry out an argument along the same lines as in the proof of
  \cite[Lemma 2.9]{CCP}. We briefly explain the idea: due to Lemma
  \ref{lem:lower-mass}, in a ball of radius $r$ around each $x_i$
  there are at least $mN$ other particles; hence there exist
  $\ell \leq \ceil{1/m}$ indices $i_1, \dots, i_\ell$ (where
  $\ceil{\cdot}$ is the ceiling function) such that
  \begin{equation*}
    \{x_1, \dots, x_N\} \subset B_{2r}(x_{i_1}) \cup \dots \cup B_{2r}(x_{i_\ell}),
  \end{equation*}
  and such that the balls $B_r(x_{i_1}),\dots,B_r(x_{i_\ell})$ are disjoint. Now,
  relabel the points $x_{i_1}, \dots, x_{i_\ell}$ so that they are
  ordered according to their first coordinate. Following the same
  argument as in Lemma \ref{lem:min-BRN} we see that
  \begin{equation*}
    |\pi_1(x_{i_k}) - \pi_1(x_{i_{k+1}})| \leq 4r + 2 R_W \quad \mbox{for all $k\in\{1,\dots,\ell-1\}$},
  \end{equation*}
  where $R_W$ is the constant in Hypothesis \ref{hyp:existence}
  (otherwise one can slightly shorten the gap in the first coordinate
  and decrease the energy). This shows that the diameter of the
  projection of the set $\{x_1, \dots, x_N\}$ in the first coordinate
  is not larger than $2 (\ceil{1/m}-1)(2r + R_W) + 4r$. As the argument can
  be repeated for all projections, we deduce that
  \begin{equation*}
    \diam\bX
    \leq 2 \sqrt{d} (\ceil{ 1/m }-1) (2r + R_W) + 4r =: K_1.
  \end{equation*}
  This holds for any $N > \bar N$. Since for $N \leq \bar N$ the
  diameter of minimisers is bounded by $K_2 := 2\sqrt{d}(\bar N - 1) R_W$ by Theorem \ref{thm:min-Rnd}, we obtain
  the result for $K := \max\{K_1, K_2\}$.
\end{proof}

\section{Many-particle limit}
\label{sec:convergence-minimisers}

In this section we complete the proof of Theorem \ref{thm:main}.

\subsection{Convergence of discrete minimisers} \label{subsec:convergence-discrete}
We show that if the
potential $W$ is unstable then any sequence of discrete minimisers has
a subsequence which converges in the narrow topology (up to
translations) to a continuum minimiser as $N \to \infty$. 

We first prove the following:
\begin{lem}\label{lem:convergence-minimiser}
	Let $W$ satisfy Hypotheses \ref{hyp:bfb}--\ref{hyp:W-close-to-0} and let it be discretely unstable. Then any sequence $(\bXN)_{N\geq2}$ of discrete minimisers converges, up to a subsequence and to translations, to some $\rho\in\P(\R^d)$. 
\end{lem}
\begin{proof}
Let
$(\bXN)_{N\geq2}$ be a sequence such that $\bXN$ is a minimiser of
$\E$ for all $N\geq2$. The diameter of $\bXN$ is uniformly bounded by
the constant $K$ in Theorem \ref{thm:discrete-diameter}.  Since $\E$
is translation invariant there exists a sequence of minimisers
of $\E$, which we still denote by $(\bXN)_{N\geq2}$, obtained by
suitable translations of the original sequence and
such that $\bXN \in (B_{K})^N$ for all $N \geq 2$. Then, since $K$ is
independent of $N$ we can extract a subsequence of $(\bXN)_{N\geq2}$ converging in the narrow
topology to a $\rho\in\P(\R^d)$.
\end{proof}

In order to complete the proof of Theorem \ref{thm:main}\eqref{it:main-bounded}
we need to show that instability (Definition \ref{defn:instability})
implies discrete instability (Definition
\ref{defn:discrete-instability}). In fact, we show that they are both
equivalent under Hypotheses \ref{hyp:bfb}--\ref{hyp:W-close-to-0}. We
also take the opportunity to compare them to the concept of
$H$-stability found in statistical mechanics; see for example
\cite[Definition 3.2.1]{Ruelle}. We actually define
$H$-\emph{in}stability, which is its complementary:

\begin{defn}[$H$-instability]
  \label{defn:H-instability}
  Let $W\:\R^d\to\R\cup\{+\infty\}$ and suppose that $W_\infty: =\lim_{|x|\to\infty} W(x)$ exists (possibly $+\infty$). We say that
  $W$ is \emph{$H$-unstable} if, for all $B\in\R$, there exist
  $N\geq2$ and $\bX\in\R^{Nd}$ with
  \begin{equation*}
    E_N(\bX) < \tfrac12 W_\infty -\tfrac{B}{2N}.
  \end{equation*}
\end{defn}

The proposition below shows that there
is equivalence among instability, discrete instability and
$H$-instability if Hypotheses \ref{hyp:bfb}, \ref{hyp:existence} and
\ref{hyp:finite} hold; equivalence between $H$-instability and instability if Hypothesis \ref{hyp:finite} holds was
already proved in \cite{SST}. If Hypotheses \ref{hyp:bfb},
\ref{hyp:existence} and \ref{hyp:beta} hold (so that a
singularity of the potential at $x=0$ is allowed), then we only have that
instability and discrete instability are equivalent and that they
imply $H$-instability. Whether the converse implication is true or not
in this case is an open question. By the proof of Proposition
\ref{prop:unstable} one sees that the main difficulty when $W$ is
unbounded is that we cannot take $B=W(0)$ in the Definition of
$H$-instability; therefore, what we can prove by our approach is only
that $H$-instability implies the complementary of strict stability.

A word on the terminology is in order: we have chosen
Definitions \ref{defn:instability} and \ref{defn:discrete-instability}
so as to maintain agreement with ``instability'' in the statistical
mechanics literature as, for example, in \cite{Bavaud}. More
importantly, we have wanted to keep ``stable'' as the opposite concept
of ``unstable'', which due to the equivalences above determines a
natural definition. Unfortunately, this terminology leaves us without
a good term to say ``there exists $\rho \in \P(\R^d)$ with
$E(\rho) \leq W_\infty/2$'', that is, to say ``$W$ is not strictly stable''.

In order to compare the concepts of stability we need to use a good
discrete approximation to a given measure $\rho$. We give it in the
following lemma, whose proof is postponed to Section
\ref{sec:gamma}. The following result actually implies the
$\Gamma$-convergence of the discrete energy to the continuum one; we
refer to Section \ref{sec:gamma} for details on this.

\begin{lem}
  \label{lem:discrete-approximation}
  Assume that the potential $W$ satisfies Hypotheses
  \ref{hyp:bfb}--\ref{hyp:W-close-to-0}.
  \begin{enumerate}
  \item \label{it:liminf-lemma} All sequences $(\bXN)_{N\geq2}$ with
    $\bXN \in \R^{Nd}$ for all $N\geq2$ such that
    $\mu_{\bXN} \rightharpoonup \rho$ in the narrow topology as $N \to \infty$ for
    some $\rho \in \P(\R^d)$ satisfy
    \begin{equation*}
      E(\rho) \leq \liminf_{N \to \infty} \E(\bXN).
    \end{equation*}

  \item \label{it:limsup-lemma} Let $\rho \in \P(\R^d)$ if Hypothesis
    \ref{hyp:finite} holds, or $\rho \in \M_p(\R^d)\cap\P(\R^d)$ with
    $p = d/ (d-\beta)$ if Hypothesis \ref{hyp:beta} holds. There
    exists a sequence $(\bXNs)_{N\geq2}$ with $\bXNs \in \R^{Nd}$ for
    all $N\geq2$ such that $\mu_{\bXNs} \rightharpoonup \rho$ in the narrow
    topology as $N \to \infty$ and
    \begin{equation*}
      E(\rho) = \lim_{N \to \infty}
      \E(\bXNs).
    \end{equation*}
    (We refer to this subsequence as a \emph{recovery sequence} for
    $\rho$.)
  \end{enumerate}
\end{lem}

Using this approximation result and Lemma
\ref{lem:convergence-minimiser} we can show the following:

\begin{prop}
  \label{prop:unstable}
  Suppose that $W$ satisfies Hypotheses
  \ref{hyp:bfb}--\ref{hyp:W-close-to-0} and
  $W_\infty: =\lim_{|x|\to\infty} W(x)$ exists (possibly
  $+\infty$). If Hypothesis \ref{hyp:finite} holds, then we have:
  \begin{equation}
    \label{eq:instability-finite}
    \mbox{$W$ is unstable $\iff$ $W$ is discretely unstable $\iff$ $W$ is $H$-unstable}.
  \end{equation}
  If Hypotheses \ref{hyp:bfb}, \ref{hyp:existence} and \ref{hyp:beta}
  hold, then we have:
  \begin{equation}
    \label{eq:instability-beta}
    \mbox{$W$ is
      unstable $\iff$ $W$ is discretely unstable $\implies$ $W$ is
      $H$-unstable}.
  \end{equation}
\end{prop}

\begin{proof}
  Let $W$ satisfy Hypothesis \ref{hyp:finite}. In this case the fact
  that instability is equivalent to $H$-instability was already proved
  in \cite[Proposition 4.1]{SST}. We therefore only have to prove that
  instability is equivalent to discrete instability. Suppose first that
  $W$ is unstable and let $\rho\in\P(\R^d)$ be such that $E(\rho)<W_\infty/2$. Then, by Lemma
  \ref{lem:discrete-approximation}\eqref{it:limsup-lemma},
  \begin{equation*}
    \lim_{N\to\infty} \E(\bXNs) = E(\rho) < \textstyle \frac12 W_\infty,
  \end{equation*}
  where $\bXNs$ is a recovery sequence for $\rho$. Therefore there
  exists $s>0$ such that $\E(\bXNs) < W_\infty/2 - s$ for all $N$
  large enough, which proves that $W$ is discretely unstable. Suppose
  now that $W$ is discretely unstable. Then there exist $s>0$ and
  $\bar N\geq2$ such that, for each $N>\bar N$, we can choose
  $\bX \in \R^{Nd}$ with $\E(\bX) < W_\infty/2 - s$ and
  $s>W(0)/(2N)$. Hence
  \begin{equation*}
    E(\mu_{\bX}) = \E(\bX) + \textstyle \frac{W(0)}{2N} < \frac{1}{2}W_\infty - s + \textstyle \frac{W(0)}{2N} < \textstyle \frac{1}{2}W_\infty,
  \end{equation*}
  which ends the proof of \eqref{eq:instability-finite}.

  Let now $W$ satisfy Hypotheses \ref{hyp:bfb}, \ref{hyp:existence}
  and \ref{hyp:beta}. Suppose that $W$ is unstable. Then we know by
  \cite[Theorem 1.4]{CCP} and Theorem \ref{thm:morrey} that there
  exists a minimiser $\rho\in\M_p(\R^d)\cap\P(\R^d)$ of $E$ with
  $p = d/(d-\beta)$. As above, Lemma
  \ref{lem:discrete-approximation}\eqref{it:limsup-lemma} gives us
  that $W$ is discretely unstable. Let $W$ be discretely unstable, so
  that there exist $s>0$ and $\bar N\geq2$ such that, for each
  $N>\bar N$ we can choose $\bX \in \R^{Nd}$ with
  $\E(\bX) < W_\infty/2 - s$. Then, by Theorem \ref{thm:min-Rnd},
  there exists a minimiser $\bXN$ of $\E$ for every $N>\bar N$ such
  that $\E(\bXN) < W_\infty/2 - s$. By Lemma
  \ref{lem:convergence-minimiser} the sequence $(\bXN)_{N\geq2}$
  converges, up to a subsequence and to translations, to some
  $\rho\in\P(\R^d)$, and Lemma
  \ref{lem:discrete-approximation}\eqref{it:liminf-lemma} gives us
  \begin{equation*}
    E(\rho) \leq \liminf_{N\to\infty} \E(\bXN) \leq \textstyle \frac12 W_\infty - s < \textstyle \frac12 W_\infty,
  \end{equation*}
  which shows that $W$ is unstable. Also, for every $B\in\R$ there
  exists $N\geq2$ large enough such that $B/(2N) < s$ which proves
  that, for such $N$,
  \begin{equation*}
    \E(\bX) < \textstyle\frac12 W_\infty - s < \textstyle\frac12 W_\infty - \textstyle\frac{B}{2N},
  \end{equation*}
  where $\bX$ is as above. This ends the proof of
  \eqref{eq:instability-beta}.
\end{proof}

We end this section with the following lemma, which finally shows
Theorem \ref{thm:main}\eqref{it:main-bounded}:

\begin{lem}
  \label{lem:convergence-minimiser-unstable}
  Let $W$ satisfy Hypotheses \ref{hyp:bfb}--\ref{hyp:W-close-to-0} and
  let it be unstable. Then any sequence $(\bXN)_{N\geq2}$ of discrete
  minimisers converges, up to a subsequence and to translations, to
  some $\rho\in\P(\R^d)$. Furthermore, $\rho$ is a continuum
  minimiser.
\end{lem}

\begin{proof}
  Let $(\bXN)_{N\geq2}$ be a sequence such that $\bXN$ is a minimiser
  of $\E$ for all $N\geq2$. By Proposition \ref{prop:unstable} and
  Lemma \ref{lem:convergence-minimiser} we know that $(\bXN)_{N\geq2}$
  converges, up to a subsequence and to translations, to some
  $\rho\in\P(\R^d)$.

  Let us prove that $\rho$ is a minimiser of $E$. Take
  $\nu\in\P(\R^d)$ if $W$ satisfies Hypothesis \ref{hyp:finite}, and
  $\nu\in \M_p(\R^d)\cap \P(\R^d)$ with $p = d/(d-\beta)$ if $W$ satisfies
  Hypothesis \ref{hyp:beta}. We know by Lemma \ref{lem:discrete-approximation}\eqref{it:limsup-lemma} that there is a recovery sequence
  $(\bXNs)_{N\geq2}$ for $\nu$. Lemma \ref{lem:discrete-approximation} and the minimality of the
  sequence $(\bXN)_{N\geq2}$ lead to
\begin{equation*}
	E(\nu) = \lim_{N\to\infty} \E(\bXNs) \geq \lim_{N\to\infty} \E(\bXN) \geq \liminf_{N\to\infty} \E(\bXN) \geq E(\rho),
\end{equation*}
which ends the proof, since, by \cite[Theorem 1.4]{CCP} and Theorem \ref{thm:morrey}, if $W$ satisfies Hypothesis \ref{hyp:beta} then minimisers of $E$ exist and belong to $\M_p(\R^d)\cap\P(\R^d)$ with $p = d/(d-\beta)$.
\end{proof}

\subsection{Unbounded growth of the diameter} 

We show that if the potential $W$ is strictly stable, then the
diameter of any sequence of discrete minimisers must diverge; that is,
we prove Theorem \ref{thm:main}\eqref{it:main-unbounded}.

We first prove that Morrey regularity is preserved under the narrow
limit. This actually further motivates the notion of
discrete Morrey measures as in Definition
\ref{defn:discrete-morrey-spaces}.
\begin{lem}
  \label{lem:morrey-limit}
  Let $(\bXN)_{N\geq2}$ be a sequence of configurations converging
  narrowly to some $\rho\in\P(\R^d)$. Let $p\in[1,\infty]$ and suppose that
  there exists $M>0$ such that $[\bXN]_{\M_p^N}\leq M$ for all
  $N\geq2$. Then $\rho\in\M_p(\R^d)$ and $\|\rho\|_{\M_p(\R^d)} \leq
  2^{d/q} M$, with $q$ the Hölder dual of $p$.
\end{lem}

\begin{proof}
  Take any integer $N \geq 2$, $x\in\R^d$ and $r>0$, and write
  $\bXN = (x_1,\dots,x_N)$.  Assume that there is $i\in\{1,\dots,N\}$
  such that $x_i\in B_r(x)$.  Since $B_r(x) \subset B_{2r}(x_i)$,
  using that $\bXN\in \M_p^N$ with $[\bXN]_{\M_p^N} \leq M$ we have
  \begin{equation*}
    \mu_{\bXN}(B_r(x))
    \leq
    \mu_{\bXN}(B_{2r}(x_i))
    \leq \textstyle\frac1N + M(2r)^{d/q}.
  \end{equation*}
  On the other hand, if there is no $i\in\{1,\dots,N\}$
  such that $x_i\in B_r(x)$ then the previous inequality holds
  trivially. By the Portmanteau
  theorem (see for example \cite[Theorem 2.1]{Billingsley}), taking limits as $N \to \infty$ gives the result:
  \begin{equation*}
    \rho(B_r(x)) \leq \liminf_{N\to\infty} \mu_{\bXN}(B_r(x))
    \leq 2^{d/q}Mr^{d/q}.
    \qedhere
  \end{equation*}
\end{proof}

We conclude by the following lemma:
\begin{lem}
  Let $W$ satisfy Hypotheses \ref{hyp:bfb}--\ref{hyp:W-close-to-0} and let it
  be strictly stable. Then any sequence $(\bXN)_{N\geq2}$ of discrete
  minimisers is such that $\diam \bXN \to \infty$ as $N\to\infty$.
\end{lem}
      
\begin{proof}
  Let $(\bXN)_{N\geq2}$ be a sequence of discrete minimisers. If
  $\diam \bXN$ does not diverge one can find, after suitable translations of $(\bXN)_{N\geq2}$ (by translation invariance of $\E$), a sequence of minimisers which are uniformly compactly supported. By compactness we can extract a
  subsequence converging in the narrow topology to some
  $\rho \in \P(\R^d)$.
	
  Let $W$ first satisfy Hypothesis \ref{hyp:finite}. Then, by the same
  argument as in the proof of Lemma
  \ref{lem:convergence-minimiser-unstable}, $\rho$ must be a continuum
  minimiser. But we know that continuum minimisers do not exist if $W$
  is strictly stable due to \cite[Theorem 3.3]{CCP} and \cite[Theorem 3.2]{SST}, so we have
  reached a contradiction. We deduce that $\diam \bXN$ diverges.
	
  If now $W$ satisfies Hypothesis \ref{hyp:beta}, then $\bXN\in\M_p^N$ with $p=d/(d-\beta)$ for all $N\geq2$, by Theorem \ref{thm:discrete-morrey}, and so $\rho\in\M_p(\R^d)$ by Lemma \ref{lem:morrey-limit}. The same
  argument as in the proof of Lemma
  \ref{lem:convergence-minimiser-unstable} gives us that $\rho$ is a minimiser of $E$ on
  $\M_p(\R^d)\cap\P(\R^d)$. This is again a contradiction since the
  results in \cite{CCP, SST} actually show that there are no minimisers on
  $\M_p(\R^d)\cap\P(\R^d)$ if $W$ is strictly stable; indeed one can construct a sequence
  $(\rho_n)_{n\in\N}$ in $\M_p(\R^d)\cap\P(\R^d)$ such that
  $E(\rho_n) \to W_\infty/2$ as $n\to\infty$, which contradicts the strict
  stability of $W$ if a minimiser on $\M_p(\R^d)\cap\P(\R^d)$ exists. (The sequence $\rho_n$ can be chosen to be the uniform
  probability on the ball of radius $n$; see \cite[Theorem 3.3]{CCP}.)
\end{proof}

\section{$\Gamma$-convergence of the discrete energy}
\label{sec:gamma}

In this section we derive a constructive way of approximating an
element of $\P(\R^d)$ by a sequence of empirical measures. We show
that this way of constructing an approximating sequence actually gives
rise to a recovery sequence with respect to our discrete and continuum
energies \eqref{eq:discrete-energy} and \eqref{eq:energy},
respectively. That is: given a measure $\rho \in \P(\R^d)$ we can
approximate it narrowly by $N$-particle empirical measures in such a
way that their discrete interaction energy also approximates the
continuum interaction energy of $\rho$. Thus we prove Lemma
\ref{lem:discrete-approximation} which was used in the previous
section. This approximation property is contained in the notion of
$\Gamma$-convergence, which we give with respect to the narrow
topology. Recall that the narrow topology on $\P(\R^d)$ is, by
definition, obtained by duality with the space of continuous bounded
functions on $\R^d$. By the Portmanteau theorem (see \cite[Theorem
2.1]{Billingsley}) it can actually be equivalently obtained by duality
with the space of Lipschitz bounded functions on $\R^d$; this is a
property which we use later on. Also, the narrow topology can be
metrised by, for example, the L\'evy--Prokhorov distance; see
\cite[Section 6]{Villani2} for the definition and other examples of
distances metrising the narrow topology. If we restrict ourselves to
elements in $\P(\R^d)$ which have finite $p$th moments for some
$p\in[1,\infty)$, then it can also be metrised by the Wasserstein
distance of order $p$ up to convergence of the $p$th moments; see \cite{AGS,Villani2}. In the following we
denote by $\sigma$ any of these metrising distances.
\begin{defn}[$\Gamma$-convergence]
  \label{defn:gamma-convergence}
  Let $A$ be a subset of $\P(\R^d)$. We say that the discrete energy
  $(\E)_{N\geq2}$ \emph{$\Gamma$-converges} (narrowly) to the
  continuum energy $E$ on $A$ if the following two inequalities are
  met for all $\rho \in A$.
  \begin{enumerate}[label=(\roman*)]
  \item \label{it:liminf} (liminf inequality) All sequences
    $(\bXN)_{N\geq2}$ with $\bXN\in\R^{Nd}$ for all $N\geq2$ such that
    $\sigma(\mu_{\bXN},\rho) \to 0$ as $N \to \infty$ satisfy
    $E(\rho) \leq \liminf_{N \to \infty} \E(\bXN)$.
  \item \label{it:limsup} (limsup inequality) There exists a sequence
    $(\bXNs)_{N\geq2}$ with $\bXNs \in \R^{Nd}$ for all $N\geq2$ such that
    $\sigma(\mu_{\bXNs},\rho) \to 0$ as $N \to \infty$ and
    $E(\rho) \geq \limsup_{N \to \infty} \E(\bXNs)$. Such a sequence is
    called a \emph{recovery sequence} for $\rho$.
  \end{enumerate}
\end{defn}
A sequence $(\bXNs)_{N\geq2}$ as in the limsup inequality is called a
recovery sequence for $\rho$ because one can check that
$\E(\bXN) \to E(\rho)$ as $N\to\infty$. The notion of
$\Gamma$-convergence arises naturally in the discrete approximation of
minimisers of energy functionals
because, along with compactness, it ensures that a sequence of
discrete minimisers converges to a minimiser of the continuum
energy. (A continuum minimiser thus exists.) Formally, if
$(\bXN)_{N\geq2}$ is a sequence of minimisers of $\E$ and there exists
$\rho\in\P(\R^d)$ such that $\sigma(\mu_{\bXN},\rho)\to0$ as
$N\to\infty$ up to a subsequence, then the $\Gamma$-convergence of
$\E$ to $E$ on a set $A$ implies: for any $\nu \in A$ there exists
$(\bYN)_{N\geq2}$ such that
\begin{equation*}
  E(\nu) \geq \limsup_{N\to\infty}
  \E(\bYN) \geq \limsup_{N\to\infty} \E(\bXN) \geq \liminf_{N\to\infty}
  \E(\bXN) \geq E(\rho),
\end{equation*}
which shows that $\rho$ is a minimiser of $E$ on $A$. This is a
fundamental theorem of $\Gamma$-convergence which we already used in
the proof of Lemma \ref{lem:convergence-minimiser-unstable}. For a
detailed introduction to $\Gamma$-convergence we refer the reader to
\cite{Braides,DalMaso}. We now show the $\Gamma$-convergence of $\E$
to $E$.

\begin{thm}\label{thm:gamma}
  Assume $W$ satisfies Hypotheses
  \ref{hyp:bfb}--\ref{hyp:W-close-to-0}.
  \begin{enumerate}
  \item If Hypothesis \ref{hyp:finite} holds, then $(\E)_{N\geq2}$
    $\Gamma$-converges to $E$ on $\P(\R^d)$.
  \item If Hypothesis \ref{hyp:beta} holds, then
    $(\E)_{N\geq2}$ $\Gamma$-converges to $E$ on $\M_p(\R^d)\cap\P(\R^d)$ with
    $p=d/(d-\beta)$.
  \end{enumerate}
\end{thm}

In the rest of this section we prove Theorem \ref{thm:gamma}; in fact,
we prove the slightly stronger statement given in Lemma
\ref{lem:discrete-approximation}. (The liminf inequality holds indeed on all of $\P(\R^d)$ even when Hypothesis \ref{hyp:beta} holds; see Remark \ref{rem:liminf}.) We first show the liminf inequality
and then the limsup inequality of Definition
\ref{defn:gamma-convergence}.

\medskip
We use the following lemma whose proof can be found in \cite[Lemma
2.1]{CCP}:
\begin{lem}
  \label{lem:lsc-E}
  If $W\:\R^d\to\R\cup\{-\infty,+\infty\}$ is bounded from below
  (resp. above) and lower (resp. upper) semicontinuous, then $E$ as
  defined in \eqref{eq:energy} is narrowly lower (resp. upper)
  semicontinuous.
\end{lem}

\subsection{Liminf inequality}\label{subsec:liminf}

Let $\rho\in\P(\R^d)$ and $(\bXN)_{N\geq2}$ be such that
$\sigma(\mu_{\bXN},\rho) \to 0$ as $N \to \infty$. Suppose first that
$W$ satisfies Hypothesis \ref{hyp:finite}. Then $W(0)$ is finite and
\begin{equation*}
  \liminf_{N\to\infty} \E(\bXN) = \liminf_{N\to\infty} \left(E(\mu_{\bXN}) - \tfrac{W(0)}{2N}\right) \geq E(\rho),
\end{equation*}
by narrow lower semicontinuity of $E$, by Lemma \ref{lem:lsc-E}, which
is the result.

Now, suppose that $W$ satisfies Hypothesis \ref{hyp:beta}; then
$W(0)=+\infty$. Assume that $\liminf_{N\to\infty} \E(\bXN)<+\infty$ or
we are done. Let $\{W_\e\}_{\e>0}$ be a family of potentials such that
$W_\e(x) = W(x)$ for all $x\in \R^d\setminus\{0\}$ and
$W_\e(0) = 1/\e$ for all $\e>0$. So defined, $W_\e$ is lower
semicontinuous; $E_\e$, defined by
$E_\e(\nu) := 2^{-1}\ird\ird W_\e(x-y) \d\nu(x)\d\nu(y)$ for all
$\nu\in\P(\R^d)$, is therefore narrow lower semicontinuous by Lemma
\ref{lem:lsc-E}. Define
$\E^\e(\bXN):=(2N^2)^{-1} \sum_{i=1}^N \sum_{j=1,j\neq i}^N
W_\e(x_i-x_j)$,
where $\bXN=(x_1,\dots,x_N)$. Then $\E^\e(\bXN) \leq \E(\bXN)$ for all
$\e>0$ and
\begin{equation}\label{eq:lsc}
	 \liminf_{N\to\infty} \E(\bXN) \geq \liminf_{N\to\infty} \E^\e(\bXN) = \liminf_{N\to\infty} \left(E_\e(\mu_{\bXN}) - \tfrac{W_\e(0)}{2N}\right) \geq E_\e(\rho).
\end{equation}
We now need to show that $E_\e(\rho) \to E(\rho)$ as $\e\to0$. If $\rho$ has no atomic part, then $E_\e(\rho) = E(\rho)$ and we are done. We want to show by contradiction that $\rho$ cannot have an atomic part. If $\rho$ has an atomic part $\alpha \delta_{z}$ for some $0<\alpha\leq1$ and $z\in\R^d$, then, by boundedness from below of $W$,
\begin{align*}
	2E_\e(\rho) &\geq \alpha^2 \ird\ird W_\e(x-y) \d \delta_{z}(x) \d \delta_{z}(y) + W_\mt{min}(1-\alpha^2)\\
	&= \frac{\alpha^2}{\e} + W_\mt{min}(1-\alpha^2) \xrightarrow[\e\to0]{}+\infty.
\end{align*}
This contradicts \eqref{eq:lsc} and the fact that $\liminf_{N\to\infty} \E(\bXN)<+\infty$. Therefore $\rho$ cannot have an atomic part and we get the result:
\begin{equation*}
	\liminf_{N\to\infty} \E(\bXN) \geq E(\rho).
\end{equation*}

\begin{rem}\label{rem:liminf}
  The computations above tell us that the liminf inequality is
  actually true on all of $\P(\R^d)$ even if Hypothesis \ref{hyp:beta}
  holds; $\rho$ does not need to be Morrey regular in the above
  proof. Hence the liminf inequality is true not only on
  $\M_p(\R^d)\cap\P(\R^d)$ (as stated in Theorem \ref{thm:gamma}), but
  also on $\P(\R^d)$ (as stated in Lemma
  \ref{lem:discrete-approximation}).
\end{rem}

\subsection{Limsup inequality}

We first assume that $\rho\in\P(\R^d)$ is compactly supported and
then we extend the result to noncompactly supported probability measures
by a density argument in Section \ref{sec:non-compact-support}. We need to construct a sequence of particle
configurations that approximates $\rho$ narrowly and whose discrete
energy approximates the continuum energy of $\rho$.

\subsubsection{Construction of the approximation}
\label{subsubsec:approx}

The construction presented here is inspired by \cite[Proposition
4.1]{SST}. Fix any $N\geq2$ and suppose that
$\supp\rho \subset [-L,L)^d$ for some $L\geq1$. Call
$n:=\floor{N^{1/(4d)}}\geq1$, where $\floor{\cdot}$ denotes the
integer part, and divide the interval $[-L,L)$ into $n$ equal
subintervals of length $2L/n$, which gives a subdivision of $[-L,L)^d$
into $n^d$ equal cubes of the form
\begin{equation*}
  \left[-L + \frac{2 i_1 L}{n}, -L + \frac{2 (i_1+1) L}{n} \right) \times \dots \times \left[-L + \frac{2 i_d L}{n}, -L + \frac{2 (i_d+1) L}{n} \right),
\end{equation*}
for $(i_1, \dots, i_d) \in \{0, \dots, n-1\}^d$. We enumerate these cubes as
$Q_i$ for each $i \in \{1, \dots, n^d\}$. In each cube $Q_i$ we place $N_i$ particles with
\begin{equation*}
  N_i := \floor{n^{4d} \rho_i}, \quad i \in \{1, \dots, n^d\},
\end{equation*}
where $\rho_i:=\rho(Q_i)$. These particles are placed at
$x_{i,1}, \dots, x_{i,N_i}$ (when $N_i = 0$, no particles are actually
placed), anywhere on different points of a square grid obtained by
subdividing the sides of $Q_i$ into $\lfloor N_i^{1/d}+1 \rfloor$
equal smaller intervals, and by taking the nodes whose coordinates are at the
centre points of these intervals. Notice that at least one of the
$\rho_i$ is larger than or equal to $1/ n^d$, so that at least one of
the $N_i$ is strictly larger than 0. We write
$$\Np:= \sum_{i=1}^{n^d} N_i,$$ the total number of particles placed so
far.  Let us write $\Ne:= N - \Np$, the number of particles that we
still need to place (with``e'' standing for ``error''). The numbers $\Np$ and
$\Ne$ should not be confused with the number $N_i$ of particles placed
in the cube $Q_i$. We observe that 
\be\label{eq:N/M}
 N -4dN^{1-1/(4d)} - N^{1/4} \leq n^{4d} - n^d = \sum_{i=1}^{n^d} (n^{4d}
\rho_i - 1) \leq \Np \leq \sum_{i=1}^{n^d} n^{4d} \rho_i = n^{4d} \leq
N, \ee
 which yields
 \begin{equation*}
   \Ne \leq 4dN^{1-1/(4d)} + N^{1/4}.
 \end{equation*}
 In particular, we see that the fraction of particles to place is
 negligible: it holds that
\begin{equation}
   \label{eq:Ne-negligible}
  \frac{\Ne}{N} \to 0 \quad \text{as $N \to +\infty$}.
\end{equation}
Of course, if $\Ne = 0$ there is nothing left to do. Otherwise, we
place the remaining $\Ne$ particles at $y_1,\dots,y_{\Ne}$ in an
auxiliary cube $[3L,3L + 1/\sqrt{d})^d$, in different nodes of a
uniform square grid with spacing
$1/ (\sqrt{d} \lfloor \Ne^{1/d} + 1 \rfloor)$. The location and size
of this auxiliary cube ensure that the distance between any particle
in the auxiliary cube and any particle in the main cube $[-L,L)^d$ is
greater than $2L$, and that the distance between any two particles in
the auxiliary cube is less than 1. The choice of the uniform grid
ensures that the Morrey regularity is kept at the discrete level; see
Lemma \ref{lem:approx-measure} below. We give mass $1/N$ to all the
particles thus placed, so that the total mass is $1$. We then define
the candidate recovery sequence for $\rho$ by gathering all particles
placed so far:
\begin{equation*}
  \bXNs :=
  (x_{1,1},\dots,x_{1,N_1},\dots,x_{n^d,1},\dots,x_{n^d,N_{n^d}},
  y_1,\dots,y_{\Ne}) \in \R^{Nd},
\end{equation*}
with the associated empirical measure
\begin{equation*}
  \mu_{\bXNs} := \theta_N \mup + (1-\theta_N) \mue,
\end{equation*}
where
\begin{equation*}
  \mup:=
  \frac{1}{\Np} \sum_{i=1}^{n^d} \sum_{k=1}^{N_i} \delta_{x_{i,k}},
  \quad \mue:= \frac{1}{\Ne} \sum_{j=1}^{\Ne} \delta_{y_j},
  \quad \mbox{and} \quad \theta_N := \frac{\Np}{N}.
\end{equation*}
Notice that $\theta_N\to1$ as $N\to\infty$, by \eqref{eq:N/M}. In the following we
refer to $x_{1,1},\dots,x_{n^d,N_{n^d}}$ as the \emph{main} particles,
and to $y_1,\dots,y_{\Ne}$ as the \emph{auxiliary} particles. In Figure \ref{fig:construction} we illustrate the above construction and we summarise the main quantities.

\bfig[!ht]
\centering
	\includegraphics[scale=0.79]{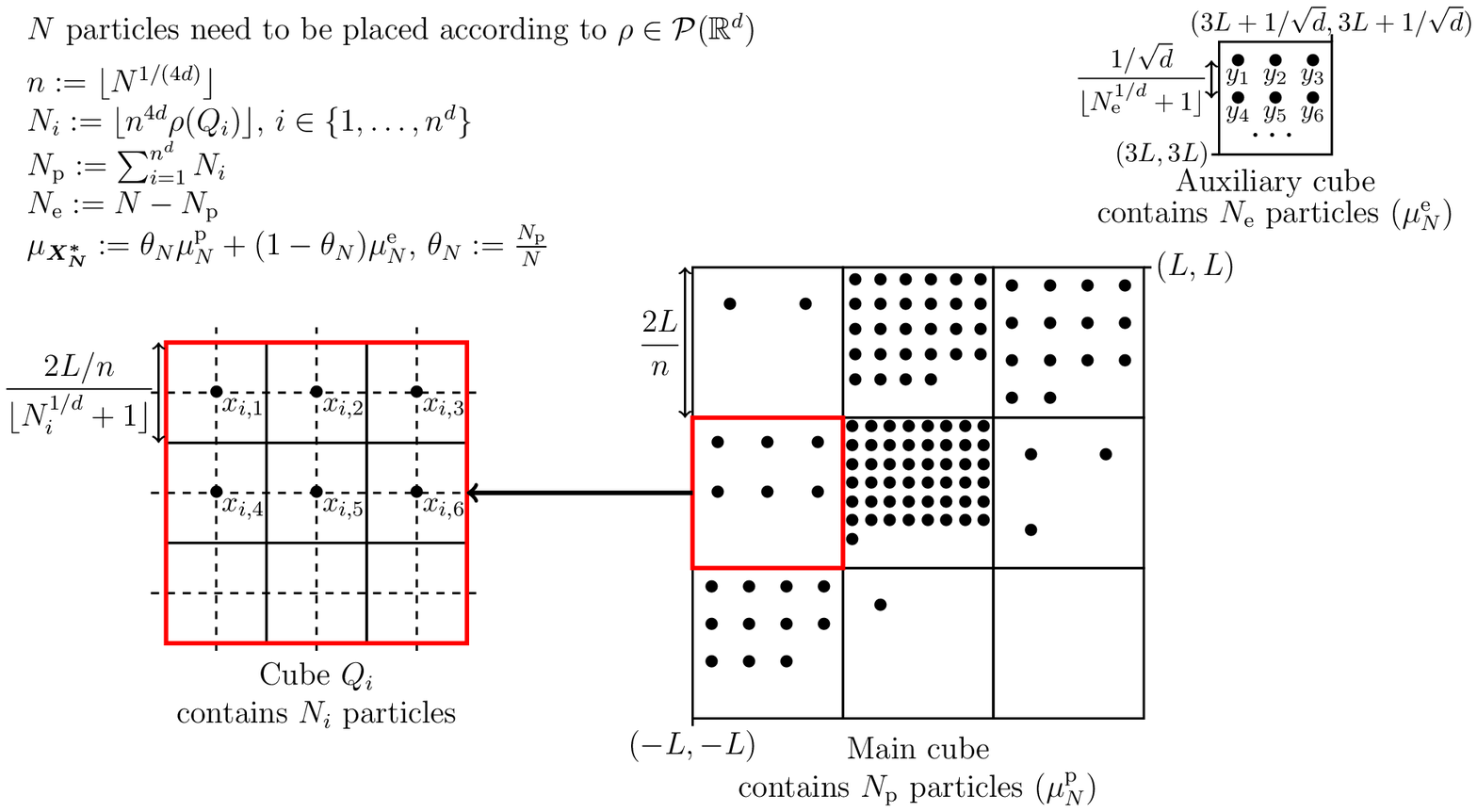}
	\caption{Schematic illustration of the construction of the empirical approximation\label{fig:construction}.} 
\efig

\subsubsection{Narrow approximation}
      
We show that $(\bXNs)_{N\geq2}$ is a good narrow approximation of
$\rho$, which is the first part of the limsup inequality in the
compactly supported case; see Definition
\ref{defn:gamma-convergence}. We also prove that if $\rho$ is
Morrey regular, so is $\bXNs$ for all $N\geq 2$.
\begin{lem}
  \label{lem:approx-measure}
  Let $\rho\in \P(\R^d)$ be compactly supported. There exists a
  sequence $(\bXNs)_{N\geq2}$ such that
  \begin{equation*}
    \sigma ( \mu_{\bXNs},\rho) \to 0 \quad \mbox{as $N\to\infty$}.
  \end{equation*}
  If furthermore $\rho \in \M_p(\R^d)$ for some $p \in [1,\infty)$, then
  $\bXNs \in \M_p^N$ for all $N\geq2$ and $[\bXNs]_{\M_p^N}$ is
  uniform in $N$.
\end{lem}

Let us point out that for a given probability density $\rho$ the
problem of finding the best empirical approximation of $\rho$ in some
topology for a fixed number of particles is called
\emph{quantisation}. Typically $\rho$ is in this context compactly
supported and the metric is the Wasserstein
distance. In this case the best approximation can be constructed by
covering the support of $\rho$ with appropriate balls and using the
Voronoi tessellation generated by their centres, and rates of convergence as $N\to\infty$ can be
obtained under suitable regularity of $\rho$; see
\cite{Kloeckner,GL}. The empirical approximation constructed in this paper is specific to our problem---we are not concerned with its
optimality in approaching $\rho$ but with the fact that it also has to preserve the energy
as $N\to\infty$; see Lemma \ref{lem:approx-bounded}.

\begin{proof}[Proof of Lemma \ref{lem:approx-measure}]
  Take $\bXNs$ as in Section \ref{subsubsec:approx}. We proceed in two
  steps.
  
  \textit{Step 1: approximation of $\rho$.} Let
  $\phi \in L^\infty(\R^d) \cap \Lip(\R^d)$ with
  $\|\phi\|_\infty \leq 1$ and $\|\phi\|_{\Lip} \leq 1$. As already
  noticed at the beginning of Section \ref{sec:gamma}, the narrow
  topology is obtained by duality with bounded Lipschitz
  functions. Hence to prove the result it suffices to prove that
  $| \int_{\R^d} \phi(x) \d \rho(x) - \int_{\R^d} \phi(x) \d
  \mu_{\bXNs}(x) |\to0$ as $N\to\infty$. First notice that
  \begin{equation*}
    \left| \int_{\R^d} \phi(x)
      \d \rho(x) - \int_{\R^d} \phi(x) \d \mu_{\bXNs}(x) \right| \leq
    \phi_N + (1-\theta_N) \left| \ird \phi(x) \d\mue(x)\right|
    \leq \phi_N + (1-\theta_N),
  \end{equation*}
  where
  $\phi_N:=| \ird \phi(x) \d \rho(x) - \theta_N \ird \phi(x) \d
  \mup(x)|$.
  Since $\theta_N\to1$ as $N\to\infty$ we only need to show that
  $\phi_N\to0$. Using that $n^{4d}-n^d\leq N_\mt{p} \leq n^{4d}$ and
  $N_i:=\floor{n^{4d}\rho_i}$ we get
  \begin{equation}\label{eq:aprx-0}
    \frac{\rho_i}{N_i + 1}
    \leq \frac{1}{\Np}
    \leq\frac{\rho_i}{N_i} + \frac{n^d\rho_i}{\Np N_i}
    \quad \mbox{for any $i \in\{1, \dots, n^d\}$}
  \end{equation}
  and obtain
  \bes
    \left| \frac{1}{\Np}  - \frac{\rho_i}{N_i} \right|
    \leq \max\left( \frac{1}{N_i + 1}, \frac{n^d}{\Np} \right)
    \frac{\rho_i}{N_i} \leq \frac{\max(1,n^d\rho_i)}{N_\mt{p}N_i} \leq \frac{n^d}{N_\mt{p}N_i}.
  \ees
  Using this, compute
  \begin{align*} \label{eq:aprx-2}
    \numberthis \phi_N
    &= \left| \sum_{i=1}^{n^d} \int_{Q_i} \phi(x) \d\rho(x)
      - \sum_{i=1}^{n^d} \frac{\theta_N}{\Np} \sum_{k=1}^{N_i}
      \phi(x_{i,k})
      \right|\\
    &\leq \left|
      \sum_{i=1}^{n^d} \int_{Q_i} \phi(x) \d \rho(x)
      - \sum_{i=1}^{n^d} \frac{\theta_N\rho_i}{N_i} \sum_{k=1}^{N_i}
      \phi(x_{i,k})
      \right|
      + \theta_N \sum_{i=1}^{n^d} \frac{n^d}{N_\mt{p}N_i}
      \sum_{k=1}^{N_i} |\phi(x_{i,k})|\\
    &\leq \left|
      \sum_{i=1}^{n^d} \frac{\rho_i}{N_i}
      \sum_{k=1}^{N_i} \left(\phi(z_i) -  \theta_N\phi(x_{i,k})
      \right)
      \right|
      + \frac{\theta_Nn^{2d}}{N_\mt{p}},
  \end{align*}
  for some $z = (z_1,\dots,z_{n^d}) \in Q_1\times\dots\times
  Q_{n^d}$.
  For the first term in \eqref{eq:aprx-2} we use that the cubes $Q_i$
  have a diameter equal to $\sqrt{d} (2L/n)$ to bound it by
  \begin{align*}
    &\sum_{i=1}^{n^d} \frac{\rho_i}{N_i} \sum_{k=1}^{N_i}
      \left| \phi(z_i) -  \phi(x_{i,k}) \right|
      + (1-\theta_N) \sum_{i=1}^{n^d} \frac{\rho_i}{N_i}
      \sum_{k=1}^{N_i}|\phi(x_{i,k})|\\
    &\phantom{{}={}}  \leq  \sum_{i=1}^{n^d} \frac{\rho_i}{N_i}
      \sum_{k=1}^{N_i} |z_i -  x_{i,k}|
      + (1-\theta_N) \leq \frac{2L \sqrt{d}}{n} + (1-\theta_N).
  \end{align*}
  Thus
  \begin{equation*}
    \phi_N \leq \frac{2L \sqrt{d}}{n}
    + (1-\theta_N) + \frac{\theta_N n^{2d}}{\Np}\leq \frac{2L \sqrt{d}}{n}
    + (1-\theta_N) + \frac{\theta_N n^{2d}}{n^{4d}-n^d},
  \end{equation*}
  using \eqref{eq:N/M}. Since $\theta_N\to1$ as
  $N\to\infty$ we can make the right-hand side above be arbitrarily
  small as $N\to\infty$, which shows the result.

  \textit{Step 2: Morrey regularity.} Assume now that $\rho$ is in
  $\M_p(\R^d)\cap\P(\R^d)$. Any cube $Q_i$ (with side of length $2L/n$) is
  contained in a ball of radius $\sqrt{d} (2L/n)$, so
  \begin{equation*}
    \rho_i:=\rho(Q_i)
    \leq \left( \frac{2 L\sqrt{d}}{n} \right)^{d/q} \|\rho\|_{\M_p(\R^d)},
  \end{equation*}
  where $q=p/(p-1)$. Hence the number of points $x_{i,k}$ on each cube
  $Q_i$ is bounded as
  \begin{equation*}
    N_i = \floor{ n^{4d} \rho_i } \leq
    n^{4d} \left( \frac{2L \sqrt{d}}{n} \right)^{d/q} \|\rho\|_{\M_p(\R^d)}.
  \end{equation*}
  Therefore the coordinate spacing $\eta$ between any two main
  particles $x_{i,k}$ satisfies
  \begin{align*}
    \eta &\geq \frac{2L/n}{\displaystyle \max_{j\in\{1,\dots,n^d\}} N_j^{1/d}+1} \geq \frac{2L/n}{2\displaystyle\max_{j\in\{1,\dots,n^d\}} N_j^{1/d}}\geq \frac{L}{n} \left(n^{4d} \left( \frac{2L \sqrt{d}}{n} \right)^{d/q} \|\rho\|_{\M_p(\R^d)} \right)^{-1/d}\\
         & = 2^{-1} n^{-4} d^{-1/(2q)} \left( \frac{2L}{n} \right)^{1/p} \|\rho\|_{\M_p(\R^d)}^{-1/d}.
  \end{align*}
  If $n=1$, that is, $N<16^d$, then there is only one main particle
  placed in $[-L,L)^d$ and we set $\eta=+\infty$ by convention, which
  trivially satisfies the above inequality. The number of main
  particles which are placed inside any ball of radius $r > 0$ centred
  at any $x_{i,k}$ can be estimated by the number of main particles
  inside a cube of side $2r$ centred at $x_{i,k}$, which is at most
  $1 + ( 2r/\eta )^d$. The total mass of $\mup$ inside that ball, not
  counting the particle at $x_{i,k}$, is therefore bounded by
  \begin{equation*}
    \mup(B_r(x_{i,k}) \setminus \{x_{i,k}\}) \leq \frac{1}{N} \left( \frac{2r}{\eta} \right)^d \leq \frac{n^{4d}4^dd^{d/(2q)}}{N}  \left( \frac{2L}{n} \right)^{-d/p} r^d \, \|\rho\|_{\M_p(\R^d)}.
  \end{equation*}
  In
  particular, when $r \leq 2L/n$, we have
  \begin{equation*}
    \mup(B_r(x_{i,k}) \setminus \{x_{i,k}\}) \leq (n^{4d}/N)4^dd^{d/(2q)} \, \|\rho\|_{\M_p(\R^d)}r^{d/q}.
  \end{equation*}
  By \eqref{eq:N/M} we have
  $n^{4d}/N \leq n^{4d}/(n^{4d} - n^d) \leq 1/(1-8^{-d})$ if $n\geq2$; if $n=1$ then $n^{4d}/N\leq 1/2\leq1/(1-8^{-d})$. Since
  $2L/n \leq 2L$ and the main and auxiliary cubes are $2L$ apart, for
  any auxiliary particle $y_j$ we get
  \begin{equation*}
    \mup(B_r(y_j)\setminus\{y_j\}) =0. 
  \end{equation*}
  For $r > 2L/n$ we need a different bound. For any ball $B_r(z)$ of
  radius $r > 0$ and centred at any $z\in\R^d$, call $I$ the set of
  indices of the cubes $Q_i$ which touch $B_r(z)$:
  \begin{equation*}
    I := \big\{i \in \{1, \dots, N\} \mid Q_i \cap B_r(z) \neq \emptyset \big\}.
  \end{equation*}
  Then
  \begin{equation*}
    \mu_N^\mt{p}(B_r(z)) \leq \sum_{i \in I} \mup(Q_i) = \sum_{i \in I} \frac{N_i}{N} \leq \sum_{i \in I} \frac{n^{4d}\rho_i}{N} = \frac{n^{4d}}{N} \rho\Big( \bigcup_{i \in I} Q_i \Big).
  \end{equation*}
  The cubes $Q_i$ have diameter $\sqrt{d}(2L/n)$, so
  $\bigcup_{i \in I} Q_i \subset B_{r + \sqrt{d}(2L/n)}$. Then, using
  that $\rho \in \M_p(\R^d)$ we obtain
  \begin{equation*}
    \mup(B_r(z)) \leq \frac{n^{4d}}{N} \left( r + \frac{2 L\sqrt{d}}{n} \right)^{d/q} \| \rho \|_{\M_p(\R^d)}.
  \end{equation*}
  Hence, for $r > 2 L /n$,
  \begin{equation*}
    \mup(B_r(z))
    \leq (n^{4d}/N) (1 + \sqrt{d})^{d/q}
    \| \rho \|_{\M_p(\R^d)} r^{d/q}.
  \end{equation*}
  Again, $n^{4d}/N \leq 1/(1-8^{-d})$.

  We now need to find a mass estimate for $\mue$. If
  $\Ne = 0$ there is nothing to estimate. Otherwise, recall that
  the auxiliary particles are positioned such that the distance
  between two closest neighbours is at least
  $1/ (\sqrt{d} \lfloor \Ne^{1/d} + 1 \rfloor)$. Take $r>0$ and
  $y_j\in\R^d$ any auxiliary particle. The number of auxiliary particles inside $B_r(y_j)$ is at most $1+2r\sqrt{d}(\Ne^{1/d} + 1)$. Thus
  \begin{equation*}
    \mue(B_r(y_j)\setminus\{y_j\})
    \leq
    \Ne^{-1}
    \left( 2r \sqrt{d} (\Ne^{1/d} + 1) \right)^d
    =
    \left( 2r \sqrt{d} (1 + \Ne^{-1/d}) \right)^d
    \leq
    4^d d^{d/2} r^{d}.
  \end{equation*}
  Since $\mu_N^\mt{e}$ is
  supported on a set of diameter $1$, we also have $\mue(B_r(y_j)\setminus\{y_j\}) \leq 4^d d^{d/2} r^{d/q}$. For any main particle $x_{i,k}$ we have
  \bes
  	\mue(B_r(x_{i,k})\setminus\{x_{i,k}\}) = 0
  \ees
  if $r\leq 2L$, and 
  \bes
  	\mue(B_r(x_{i,k})\setminus\{x_{i,k}\}) \leq \mue(B_r(y_j)\setminus\{y_j\}) \leq 4^d d^{d/2} \leq 4^d d^{d/2} r^{d/q}
  \ees
  for any auxiliary particle $y_j$ if $r>2L\geq 2 > \diam(\supp\mue)=1$.
  
  All in all we have shown that there exist $M_\mt{p}>0$ and $M_\mt{e}>0$ such that, for any
  $z$ in the set
  $\{x_{1,1},\dots,x_{n^d,N_{n^d}},y_1,\dots,y_{\Ne}\}$ and
  $r>0$,
  \begin{equation*}
    \mu_{\bXNs}(B_r(z)) - \tfrac1N = \mu_{\bXNs}(B_r(z)\setminus\{z\})
    \leq \theta_N M_\mt{p} r^{d/q} + (1-\theta_N) M_\mt{e} r^{d/q}.
  \end{equation*}
  Since $\theta_N\leq1$ this ends the proof.
\end{proof}

\subsubsection{Approximation of the energy}
\label{subsubsec:aprx}

We show that $(\bXNs)_{N\geq2}$ gives rise to a good approximation of
the continuum energy $E$, which is the second part of the limsup
inequality in the compactly supported case; see Definition
\ref{defn:gamma-convergence}. (Equivalently, we show Lemma \ref{lem:discrete-approximation}\eqref{it:limsup-lemma} in the compactly supported case---notice that the liminf and limsup inequalities together actually show
the convergence of the energy, as stated in Lemma
\ref{lem:discrete-approximation}\eqref{it:limsup-lemma}).

\begin{lem}
  \label{lem:approx-bounded}
  Suppose that $W$ satisfies Hypotheses
  \ref{hyp:bfb}--\ref{hyp:W-close-to-0}. Take $\rho \in \P(\R^d)$ if
  $W$ satisfies Hypothesis \ref{hyp:finite}, or $\rho \in \M_p(\R^d)\cap\P(\R^d)$
  with $p=d/(d-\beta)$ if $W$ satisfies Hypothesis
  \ref{hyp:beta}. Assume in any case that $\rho$ has compact
  support. Then
  \begin{equation*}
    \lim_{N\to\infty} \E(\bXNs) = E(\rho).
  \end{equation*}
\end{lem}

\begin{proof}
  Notice that $E(\rho) < +\infty$ by boundedness of $W$ or Lemma
  \ref{lem:bound-potential-beta-ball}. Take $\bXNs$ as in Section
  \ref{subsubsec:approx} and Lemma \ref{lem:approx-measure}. Assume
  first that $W$ satisfies Hypothesis \ref{hyp:finite}. By Lemma
  \ref{lem:lsc-E}, $E$ is both upper and lower semicontinuous in the
  narrow topology (and hence continuous at any $\rho$ with $E(\rho)$
  finite); also, Lemma \ref{lem:approx-measure} tells us that
  $\sigma(\mu_{\bXNs}, \rho) \to 0$ as $N \to \infty$. Hence the
  result:
  \begin{equation*}
    \lim_{N\to\infty} E_N(\bXNs)
    = \lim_{N\to\infty} \left(E(\mu_{\bXNs})
      - \textstyle\frac{W(0)}{2N}\right) = E(\rho).
  \end{equation*}
  Assume now that $W$ satisfies Hypothesis \ref{hyp:beta}. Arranging
  the terms in 1) interactions among the $N_{\mt{p}}$ particles in the
  main cube, 2) interactions between particles in the main cube and
  particles in the auxiliary cube and 3) interactions among particles
  in the auxiliary cube, we have, by
  bilinearity of $\E$,
  \begin{align*}\label{eq:S}
    \numberthis \;\;\;\; & 2 \left| E(\rho) - E_N(\bXNs) \right|
    \\
    & \leq \Bigg|
      \sum_{i,j=1}^{n^d} \int_{Q_i} \int_{Q_j} W(x-y) \d\rho(x) \d\rho(y)
      - \frac{1}{N^2} \underset{(i,k) \neq (j,\ell)}
      {\sum_{i,j=1}^{n^d} \sum_{k = 1}^{N_i} \sum_{\ell=1}^{N_j}}
      W(x_{i,k} - x_{j,\ell})
      \Bigg|
    \\
    &\phantom{{}={}}
      + \frac{2}{N^2} \left| \sum_{i=1}^{n^d} \sum_{k = 1}^{N_i}
      \sum_{j=1}^{\Ne} W(x_{i,k} - y_j)\right|
      + \frac{1}{N^2} \left| \sum_{i=1}^{\Ne}
      \sum_{\substack{j=1\\j\neq i}}^{\Ne}
    W(y_i - y_j)
    \right|
    \\
    &=: S_1 + S_2 + S_3.
  \end{align*}
  We break the $i$- and $j$-sums in $S_1$ into two sets: the set
  of $i,j$ such that $Q_i$ and $Q_j$ are far apart and its
  complement. For $\eta > 0$ we define
  \begin{equation*}
    I_\eta :=
    \left\{ (i,j) \in \{1,\dots,n^d\}^2
      \mid  \dist(Q_i, Q_j) > \eta \right\},
  \end{equation*}
  and we call $I_\eta^\mt{c}$ its complement in
  $\{1,\dots,n^d\}^2$. Pick $\eta$ small enough and $n$ large enough
  such that $|x-y|\leq1$ for all $(x,y)\in Q_i\times Q_j$ and
  $(i,j)\in I_\eta^\mt{c}$. We get
  \begin{align*}
    \label{eq:S0N-1}
    \numberthis \;\;\;\;
    & \left| \sum_{(i,j) \in I_\eta^\mt{c}}
      \int_{Q_i} \int_{Q_j} W(x-y) \d\rho(x) \d\rho(y)
      - \frac{1}{N^2} \underset{(i,k) \neq (j,\ell)}
      {\sum_{(i,j)  \in I_\eta^\mt{c}}
      \sum_{k = 1}^{N_i} \sum_{\ell=1}^{N_j}}
      W(x_{i,k} - x_{j,\ell}) \right|
    \\
    &\leq
      C_W\sum_{(i,j) \in I_\eta^\mt{c}}
      \int_{Q_i} \int_{Q_j} |x-y|^{2-\beta} \d\rho(x) \d\rho(y)
      + \frac{C_W}{N^2} \underset{(i,k) \neq (j,\ell)}
      {\sum_{(i,j)  \in I_\eta^\mt{c}} \sum_{k = 1}^{N_i}
      \sum_{\ell=1}^{N_j}} |x_{i,k} - x_{j,\ell}|^{2-\beta}
    \\
    &\leq C_WC_\eta ( \|\rho\|_{\M_p(\R^d)}
      + [\mu_N^\mt{p}]_{\M_p^N})
      \leq C_WC_\eta ( \|\rho\|_{\M_p(\R^d)} + M_\mt{p}),
  \end{align*}
  where $C_\eta$ is a quantity such that $C_\eta \to 0$ as $\eta\to0$,
  as can be deduced from Lemmas \ref{lem:bound-potential-beta-ball}
  and \ref{lem:bound-potential-beta-ball-discrete}, using that
  $\rho\in\M_p(\R^d)$ and $\mu_N^\mt{p}\in\M_p^N$ with
  $[\mu_N^\mt{p}]_{\M_p^N}\leq M_\mt{p}$ for some $M_\mt{p}>0$; see
  Step 2 in the proof of Lemma \ref{lem:approx-measure}.

  % For the terms $(i,j) \in I_\eta$ we use \eqref{eq:aprx-0} and
  % $N_\mt{p}\leq n^{4d}$ to get, for all $i,j\in\{1,\dots,n^d\}$,
  % \begin{align}
  %   \label{eq:aprx-1}
  %   \numberthis
  %   \left| \frac{1}{N_\mt{p}^2} - \frac{\rho_i\rho_j}{N_iN_j} \right|
  %   &\leq \max\left( \frac{N_i+N_j+1}{(N_i+1)(N_j+1)},\left(2+\frac{n^d}{N_\mt{p}}\right)\frac{n^d}{N_\mt{p}} \right)\frac{\rho_i\rho_j}{N_iN_j}\\
  %   &\leq \frac{\max\left(N_i+N_j+1,(2N_\mt{p}+n^d)n^d\rho_i\rho_j\right)}{N_\mt{p}^2N_iN_j} \leq \frac{3n^{5d}}{N_\mt{p}^2N_iN_j}
  %     \end{align}
  % and obtain
  For the terms $(i,j) \in I_\eta$ we have
  \begin{multline*}
    \Big| \sum_{(i,j) \in I_\eta}
      \int_{Q_i} \int_{Q_j} W(x-y) \d\rho(x) \d\rho(y)
      - \frac{1}{N^2} \underset{(i,k) \neq (j,\ell)}
      {\sum_{(i,j)  \in I_\eta}
        \sum_{k = 1}^{N_i} \sum_{\ell=1}^{N_j}}
      W(x_{i,k} - x_{j,\ell})
    \Big|
    \\
    =
    \Big|
      \sum_{(i,j) \in I_\eta} \int_{Q_i} \int_{Q_j}
      \frac{1}{N_{i} N_j} \sum_{k = 1}^{N_i} \sum_{\ell=1}^{N_j}
      \left( W(x-y)
        - \frac{N_i N_j}{N^2 \rho_i \rho_j} W(x_{i,k} - x_{j,\ell})
      \right)
      \d\rho(x) \d\rho(y)
    \Big|
    \\
    \leq
    \sum_{(i,j)  \in I_\eta} \int_{Q_i} \int_{Q_j}
      \frac{1}{N_{i} N_j} \sum_{k = 1}^{N_i} \sum_{\ell=1}^{N_j}
      \left| W(x-y)
        - W(x_{i,k} - x_{j,\ell})
      \right|
      \d\rho(x) \d\rho(y)
    \\
    +
    \sum_{(i,j)  \in I_\eta}
    \frac{1}{N_{i} N_j} \sum_{k = 1}^{N_i} \sum_{\ell=1}^{N_j}
    \left|
       \rho_i \rho_j - \frac{N_i N_j}{N^2}
    \right|
    | W(x_{i,k} - x_{j,\ell}) |
    =:
    S_{1,1} + S_{1,2},
  \end{multline*}
  % \begin{align*}\label{eq:S0N-2}
  %   \numberthis &\left| \sum_{(i,j) \in I_\eta} \int_{Q_i} \int_{Q_j} W(x-y) \d\rho(x)\d \rho(y) - \frac{\theta_N^2}{N_\mt{p}^2} \sum_{(i,j)  \in I_\eta} \sum_{k = 1}^{N_i} \sum_{\ell=1}^{N_j} W(x_{i,k} - x_{j,\ell})  \right| \\
  %               &\leq \Bigg| \sum_{(i,j)  \in I_\eta} \int_{Q_i} \int_{Q_j} W(x-y) \d\rho(x) \d\rho(y)\\
  %               &\phantom{{}={}} - \theta_N^2\sum_{(i,j)  \in I_\eta}  \int_{Q_i} \int_{Q_j} \frac{1}{N_i} \sum_{k = 1}^{N_i}  \frac{1}{N_j} \sum_{\ell=1}^{N_j} W(x_{i,k} - x_{j,\ell}) \d\rho(x) \d\rho(y) \Bigg| \\
  %               &\phantom{{}={}}+  \theta_N^2 \sum_{(i,j)  \in I_\eta}^{n^d} \sum_{k = 1}^{N_i} \sum_{\ell=1}^{N_j} \frac{3n^{5d}}{N_\mt{p}^2N_iN_j} |W(x_{i,k} - x_{j,\ell})|=: R_{0,N} + R_{1,N}
  % \end{align*}
  where we recall that $\rho_i := \rho(Q_i)$.  We show now that
  $S_{1,1}$ and $S_{1,2}$ become small as $N \to +\infty$. For
  $S_{1,1}$,
%   \begin{align*}
%     \label{eq:S0N-3}
%     \numberthis \;\;\;\; S_{1,1} &\leq \sum_{(i,j)\in I_\eta}^{n^d}  \int_{Q_i} \int_{Q_j} \frac{1}{N_i} \sum_{k = 1}^{N_i} \frac{1}{N_j} \sum_{\ell=1}^{N_j} \left| W(x-y) - \theta_N^2W(x_{i,k} - x_{j,\ell}) \right| \d\rho(x)\d \rho(y)\\
% 	&\leq \sum_{(i,j)\in I_\eta}^{n^d}  \int_{Q_i} \int_{Q_j} \frac{1}{N_i} \sum_{k = 1}^{N_i} \frac{1}{N_j} \sum_{\ell=1}^{N_j} \left| W(x-y) - W(x_{i,k} - x_{j,\ell}) \right| \d\rho(x)\d \rho(y)\\
% 	&\phantom{{}={}} + (1-\theta_N^2) \sum_{(i,j)\in I_\eta}^{n^d}  \frac{\rho_i\rho_j}{N_iN_j} \sum_{k = 1}^{N_i} \sum_{\ell=1}^{N_j} \left| W(x_{i,k} - x_{j,\ell}) \right| =: R_{0,N}'+R_{0,N}''.
% \end{align*}
  the terms $x-y$ and $x_{i,k}-x_{j,\ell}$ satisfy
  $\eta\leq |x - y| \leq 2L\sqrt{d}$ and
  $\eta\leq |x_{i,k}-x_{j,\ell}| \leq 2L\sqrt{d}$. Thus, since
  $W\in\C^1(\R^d\setminus\{0\})$, there exists $W_\eta'>0$ with
  $|W(x-y)-W(x_{i,k}-x_{j,\ell})| \leq (|x-x_{j,k}| +
  |y-x_{i,\ell}|)W_\eta'$,
  and, since the diameter of any cube $Q_i$ is $\sqrt{d}(2L/n)$,
  % ***** following eq needs slight modification
  \begin{equation}
    \label{eq:S0N-3.5}
    S_{1,1} \leq \frac{4W_\eta'L\sqrt{d}}{n}.
  \end{equation}
  The terms $x_{i,k}-x_{j,\ell}$ in $S_{1,2}$ also verify
  $\eta\leq |x_{i,k}-x_{j,\ell}| \leq 2L\sqrt{d}$, and so there exists
  $W_\eta>0$ such that $|W(x_{i,k}-x_{j,\ell})| \leq W_\eta$. Hence
  \begin{multline}
    \label{eq:S0N-4}
    S_{1,2} \leq W_\eta
    \sum_{(i,j)  \in I_\eta}
    \left|
      \rho_i \rho_j - \frac{N_i N_j}{N^2}
    \right|
     = W_\eta
    \sum_{(i,j)  \in I_\eta}
    \left(
      \rho_i \rho_j - \frac{N_i N_j}{N^2}
    \right)
    \\
    \leq
    W_\eta
    \sum_{i,j = 1}^{n^d}
    \left(
      \rho_i \rho_j - \frac{N_i N_j}{N^2}
    \right)
    = W_\eta \left( 1 - \frac{N_{\mt{p}}^2}{N^2}\right)
    =  W_\eta \left( 1 - \frac{N_{\mt{p}}}{N}\right)
    \left( 1 + \frac{N_{\mt{p}}}{N}\right)
    \leq
    2 W_\eta \frac{N_{\mt{e}}}{N}
    .
  \end{multline}
  We notice that $N_{\mt{e}} / N \to 0$ as $N \to +\infty$; see
  \eqref{eq:Ne-negligible}. Letting $n\to\infty$ (that is,
  $N\to\infty$) and then $\eta\to0$ in this order in \eqref{eq:S0N-1},
  \eqref{eq:S0N-3.5} and \eqref{eq:S0N-4} gives that $S_1 \to 0$ as
  $N\to\infty$.

  We now deal with terms $S_2$ and $S_3$ in \eqref{eq:S}. As the terms
  $x_{i,k}-y_j$ in $S_2$ satisfy
  $2L\leq |x_{i,k} - y_j|\leq 5L\sqrt{d}$ we have that
  $|W(x_{i,k}-y_j)|\leq 2 W_L$ for some $W_L>0$ and
  \begin{equation*}
    S_2 \leq 4W_L\frac{N_{\mt{e}}}{N} .
  \end{equation*}
  By Step 2 of the proof of Lemma \ref{lem:approx-measure} we know
  that $\mue \in \M_p^N$ with $[\mue]_{\M_p^N} \leq M_\mt{e}$ for some
  $M_\mt{e}>0$. Also, the terms $y_i-y_j$ in $S_3$ verify
  $|y_i-y_j|\leq 1$ and, by Hypothesis \ref{hyp:beta} and Lemma
  \ref{lem:bound-potential-beta-ball-discrete}, we get
  \begin{equation*}
    S_3 \leq C_W C_d M_\mt{e} \frac{N_{\mt{e}}^2}{N^2} ,
  \end{equation*}
  for some constant $C_d > 0$. Clearly we have $S_2 \to 0$ and
  $S_3 \to 0$ as $N\to\infty$ since ${N_{\mt{e}}}/{N} \to 0$,
  which concludes the proof.
\end{proof}

\subsubsection{Extension to noncompactly supported probability
  measures}
\label{sec:non-compact-support}

We extend Lemmas \ref{lem:approx-measure} and \ref{lem:approx-bounded}
to the case when $\rho$ is not necessarily compactly supported; this
finishes the proof of Lemma \ref{lem:discrete-approximation} and
Theorem \ref{thm:gamma}.

We proceed by density. Take $\rho\in\P(\R^d)$ if Hypothesis
\ref{hyp:finite} holds, or $\rho \in \M_p(\R^d)\cap\P(\R^d)$ with
$p = d / (d-\beta)$ if Hypothesis \ref{hyp:beta} holds. Let
$(\rho_l)_{l>0}$ be a sequence of compactly supported probability
measures such that $\sigma(\rho_l,\rho) \to0$ and
$E(\rho_l)\to E(\rho)$ as $l\to\infty$; for example take $\rho_l$ to
be the normalisation of $\rho$ restricted to the ball $B_l$. By Lemmas
\ref{lem:approx-measure} and \ref{lem:approx-bounded} we can construct a sequence of particles
$(\boldsymbol{X_{N,l}^*})_{N\geq2}$ such that
$\sigma(\mu_{\boldsymbol{X_{N,l}^*}},\rho_l)\to0$ and
$\E(\boldsymbol{X_{N,l}^*})\to E(\rho_l)$ as $N\to\infty$ for any
$l>0$. Therefore, for any subsequence
$(\boldsymbol{\bY_{k,l}})_{k\in\N}:=(\boldsymbol{X_{N_k,l}^*})_{k\in\N}$
we have $\sigma(\mu_{\boldsymbol{\bY_{k,l}}},\rho_l)\to0$ and
$F_k(\boldsymbol{\bY_{k,l}})\to E(\rho_l)$ as $k\to\infty$, where we
write $F_k$ for $E_{N_k}$. By the triangle inequality,
\begin{equation*}
  \sigma(\mu_{\boldsymbol{\bY_{k,l}}},\rho)
  \leq \sigma(\mu_{\boldsymbol{\bY_{k,l}}},\rho_l) + \sigma(\rho_l,\rho).
\end{equation*}
Therefore, for any $l>0$ there exists $k(l)\in\N$ such that
$k(l) \to\infty$ as $l\to\infty$, and
\begin{equation*}
	\sigma(\mu_{\boldsymbol{Y_{k(l),l}}},\rho) \leq \tfrac1l + \sigma(\rho_l,\rho) \to 0 \quad \mbox{and} \quad F_{k(l)}(\boldsymbol{Y_{k(l),l}}) \leq \tfrac1l + E(\rho_l) \to E(\rho)
\end{equation*}
as $l\to\infty$. This, together with the liminf inequality shown in Section \ref{subsec:liminf}, proves that the subsequence $(F_{k(l)})_{l>0}$
$\Gamma$-converges to $E$ as $l\to\infty$. For any subsequence
$(F_k)_{k\in\N} = (E_{N_k})_{k\in\N}$ we can therefore extract a
further subsequence which $\Gamma$-converges to $E$, which in turn
shows that $(\E)_{N\geq2}$ $\Gamma$-converges to $E$ by the Urysohn
property of the $\Gamma$-convergence; see \cite[Proposition
1.44]{Braides}.

\subsection*{Acknowledgements}

J.~A.~Cañizo was supported by the Spanish \emph{Ministerio de
  Eco\-no\-mía y Competitividad} and the European Regional Development
Fund (ERDF / FEDER), project MTM2014-52056-P. F. S. Patacchini thanks Imperial College London for supporting his PhD studies via a Roth studentship. The authors wish to thank
the Mittag-Leffler Institute for their support during the programme
``Interactions between Partial Differential Equations \& Functional
Inequalities'' from 1 September to 16 December 2016, when
progress on this paper was made. The authors are grateful to Jos\'e
Antonio Carrillo for his insightful comments on earlier versions of
this paper.

\bibliography{particle_interaction_energy}
\bibliographystyle{abbrv}

%================== END DOCUMENT =====================================
\end{document}